\documentclass[reqno,a4paper,11pt]{article}
\usepackage{amsmath,amsthm,dsfont,amsfonts,amssymb,fancyhdr}
\usepackage{enumerate}
\usepackage{bbm,color,soul,supertabular,longtable,verbatim,extarrows}
\usepackage{titlesec}
\usepackage{graphicx}
\usepackage{subcaption}
\usepackage[numbers]{natbib}
\usepackage{multirow,bm,rotating}
\usepackage{booktabs,makecell}
\usepackage{tikz}
\usepackage[noblocks]{authblk}
\usepackage{ulem}
\usepackage{cancel}
\usepackage[top=2.4cm,bottom=2.2cm,left=2.6cm,right=2cm]{geometry} 
\usepackage{indentfirst} 
\usepackage{float} 
\usepackage[linesnumbered,ruled, vlined]{algorithm2e}
\usepackage{setspace}
\usepackage{array}
\usepackage{chngcntr} 

\newtheorem{theorem}{Theorem}[section] 
\newtheorem{corollary}[theorem]{Corollary}

\newtheorem{definition}[theorem]{Definition}
\newtheorem{lemma}[theorem]{Lemma}

\begin{document}
\normalem
\title{\textbf{Neighbor Connectivity of Undirected Toroidal Meshes}
\footnote{This research is supported by Beijing Natural Science Foundation (1252010) and the Fundamental Research Funds for the Central Universities (2024XJJS18).}
}

\author{Hui-Ming Huang$^{\rm a}$, \quad  Ruichao Niu$^{\rm a}$\footnote{Corresponding author. \newline {\em E-mail address:} huanghm@muc.edu.cn (H.-M. Huang), niuruichao@muc.edu.cn (R. Niu), xum@bnu.edu.cn(M. Xu), spade@ntub.edu.tw(J.-M. Chang).}, \quad  Min Xu$^{\rm b}$, \quad Jou-Ming Chang $^{\rm c}$ 
\\
{\footnotesize   \em $^{\rm a}$ College of Science, Minzu University of China, Beijing, 100081, China
\\
\em $^{\rm b}$ School of Mathematical Sciences, Beijing Normal University, Laboratory of Mathematics and Complex Systems, Ministry of Education, Beijing, 100875, China
\\
\em $^{\rm c}$ Institute of Information and Decision Sciences, National Taipei University of Business, Taipei, Taiwan
}
}
\date{}

\maketitle

\renewcommand{\abstractname}{Abstract}
\begin{abstract}
In this paper, we examine the neighbor connectivity, denoted as $\kappa_{NB}$, of the undirected toroidal mesh $C(d_1,d_2,\ldots,d_n)$. We demonstrate that $\kappa_{NB}(C(d_1,d_2,\ldots,d_n)) = n$ for all $n \ge 2$ and $d_i \ge 3$ (for $1 \le i \le n$). Additionally, we perform a computer simulation experiment on neighbor connectivity in undirected toroidal meshes. This experiment not only supports our theoretical findings with empirical results but also provides a deeper understanding of neighbor structure failures in undirected toroidal meshes.
\\
\\
{\bf keywords}: undirected toroidal mesh; neighbor connectivity; fault-tolerance.

\end{abstract}

\section{Introduction}
In a multiprocessor system, processors are usually connected through a specific interconnection network. However, as the system increases in size, it becomes inevitable that some processors will fail. To maintain proper functionality, healthy processors need to remain connected. Therefore, it is essential for the interconnection network of a multiprocessor system to have strong fault tolerance.

In practice, researchers usually employ simple undirected graphs to model these networks. Let $G$ be a simple undirected graph with vertex set $V(G)$ and edge set $E(G)$. If there is an edge $e = (u, v)$ in $E(G)$ connecting two distinct vertices $u$ and $v$ in $V(G)$, we say that $e$ is \textit{incident with} $u$ and $v$, and $u$ and $v$ are \textit{adjacent}. For a vertex $u\in V(G)$, the set of all vertices adjacent to $u$ in $G$ is called the \textit{open neighborhood} of $u$, denoted by $N_G(u)$ (or simply $N(u)$ for short, and subsequent notations are abbreviated similarly). The \textit{closed neighborhood} of $u$ in $G$ is denoted as $N_G[u] = N_G(u) \cup \{u\}$. For a subset $U \subseteq V(G)$, the open neighborhood of $U$ in $G$ is denoted by $N_G(U) = \bigcup\limits_{u \in U}N_G(u) \setminus U$, and the closed neighborhood of $U$ in $G$ is denoted by $N_G[U] = N_G(U) \cup U$. The subgraph of $G$ induced by $U$ is represented as $G[U]$. The induced subgraph $G[V(G) \setminus U]$ is commonly abbreviated as $G-U$.

The classic concept of connectivity in graph theory is an important parameter for measuring the fault tolerance of networks. For a graph $G$, the \textit{connectivity} is denoted by $\kappa(G)$ and is defined as the minimum number of vertices in any subset $U \subseteq V(G)$ such that $G - U$ results in a disconnected graph or a trivial graph (i.e., a graph with only one vertex). When using connectivity to measure the fault tolerance of a network, we typically assume that failures are limited to the faulty processor itself and do not affect neighboring processors.  However, in many cases, failures tend to spread, i.e., the failure of one processor may also lead to the failure of its neighboring processors. This highlights certain limitations of classic connectivity. To address these limitations, we can utilize an extended concept of connectivity known as neighbor connectivity. It takes into account the impact of faulty processors on their neighboring processors, thereby providing a more accurate measure of the fault tolerance of networks.

The concept of neighbor connectivity originated from Gunther and Hartnell's research on spy networks~\cite{gunther1985neighbour,gunther1978minimizing,gunther1980optimal}. For a graph $G$, the \textit{neighbor connectivity}, denoted by $\kappa_{NB}(G)$, is defined as the minimum cardinality of all sets $U$ that meet the following conditions: $U \subseteq V(G)$, and $G \ominus U$ is a disconnected, complete, or empty graph (i.e., a graph with no vertices), where $G \ominus U = G - N[U]$ is called the \textit{survival graph} of $G$ for $U$. When $U = \{u\}$, we simply write $G \ominus U$ as $G \ominus u$.

Researchers have conducted numerous studies on neighbor connectivity. Gunther et al.~\cite{gunther1987neighbor} proved that $\kappa_{NB}(G) \le \kappa(G)$ for any graph $G$. Doty et al.~\cite{doty1996cayley} showed that for a given graph $G$ and an integer $k$, determining whether $\kappa_{NB}(G) \le k$ is an NP-complete problem. They also investigate the algebraic properties of the generating set that characterize Cayley graphs with a neighbor connectivity of $1$. Furthermore, Doty \cite{doty2006new} improved the upper bound of the neighbor connectivity of Abelian Cayley graphs. Table \ref{the neighbor connectivity of some networks} presents known results regarding the neighbor connectivity for some specific networks.

\renewcommand{\arraystretch}{1.2}
\begin{table}[htbp]\small
	\caption{Known results of some networks on neighbor connectivity}
	\label{the neighbor connectivity of some networks}
	\centering
	\begin{tabular}{llc}
		\toprule
		Network $G$ & Neighbor connectivity $\kappa_{NB}(G)$ & References \\
		\midrule
		Product graph $K_m \times K_n$ & $\min \{m-1,n-1\}$, $(m,n \ge 3)$ & \cite{gunther1991m} \\
		Hypercube $Q_n$ & $\lceil \frac{n}{2} \rceil$, $(n \ge 2)$ & \cite{dvovrak2020neighbor} \\
		Alternating group network $AN_n$ & $n-1$, $(n \ge 4)$ & \cite{shang2018neighbor} \\
		Alternating group graph $AG_n$ & $n-2$, $(n \ge 5)$ & \cite{abdallah2021neighbor} \\
		Star graph $S_n$ & $n-1$, $(n \ge 3)$ & \cite{shang2018neighbor} \\
		$k$-ary $n$-cube $Q_n^k$ & $n$, $(n \ge 2, k \ge 3)$ & \cite{dvovrak2020neighbor} \\
		Locally twisted cube networks $LTQ_n$ & $\lceil \frac{n}{2} \rceil$, $(n \ge 2)$ & \cite{kung2022neighborhood} \\
		Pancake network $P_n$ & $n-1$, $(n \ge 3)$ & \cite{gu2023neighbor} \\
		Burnt pancake network $BP_n$ & $n$, $(n \ge 2)$ & \cite{gu2023neighbor} \\
		Hierarchical star network $HS_n$ & $n-1$, $(n \ge 3)$ & \cite{gu2023subversion} \\
		Complete cubic network $CC_n$ & $\lceil \frac{n}{2} \rceil + 1$, $(n \ge 2)$ & \cite{gu2023subversion} \\
		Hierarchical cubic network $HC_n$ & $\lceil \frac{n}{2} \rceil + 1$, $(n \ge 2)$ & \cite{gu2023subversion} \\
		\bottomrule
	\end{tabular}
\end{table}

In this paper, we study the neighbor connectivity of undirected toroidal meshes. The rest of this paper is organized as follows. Section~\ref{preliminaries} presents some fundamental concepts relevant to this work. Section~\ref{Sec:main} establishes the main result of the neighbor connectivity for undirected toroidal meshes. Section~\ref{simulation} takes experiments to simulate the occurrence of failures in undirected toroidal meshes, followed by an analysis of the related phenomena concerning the neighborhood connectivity of these networks based on experimental results and discussions. Section~\ref{concluding remarks} contains our concluding remarks.

\section{Preliminaries} \label{preliminaries}

Let $n$ be a positive integer. A \textit{path} $P$ of length $n-1$ is represented by a vertex sequence $\langle v_1,v_2,\ldots,v_n \rangle$, where $v_1$ and $v_n$ are the two endpoints of $P$. As usual, this path is referred to as a \textit{$(v_1,v_n)$-path} and denoted by $P_{v_1 v_n}$. Notably, when $n=1$, $P$ consists of a single vertex (also known as a \textit{singleton}), indicating that the length of $P$ is 0. For two paths $P_{v_1 v_n}=\langle v_1,v_2,\ldots,v_n \rangle$ and $P_{u_1 u_n}=\langle u_1,u_2,\ldots,u_n \rangle$, we say that they are \textit{vertex-disjoint} (or simply \textit{disjoint}) if the sets of their vertices do not intersect, i.e., $V(P_{v_1 v_n}) \cap V(P_{u_1 u_n}) = \varnothing$. In particular, if $v_n$ and $u_1$ are adjacent, we can concatenate these two paths into a longer one, represented as
\[
\langle v_1,v_2,\ldots,v_n,u_1,u_2,\ldots,u_n \rangle,
\]
which can be denoted by $P_{v_1 u_n}=\langle P_{v_1 v_n},P_{u_1 u_n} \rangle$. Especially, if $P_{v_1 v_n}$ (resp.\ $P_{u_1 u_n}$) is a singleton, the concatenated path is simplified as $\langle v_1, P_{u_1 u_n} \rangle$ (resp.\ $\langle P_{v_1 v_n}, u_1 \rangle$). Similarly, this notion of concatenation can be extended to three or more paths.

Given two sets $X$ and $Y$, a path $P$ is called an \textit{$(X, Y)$-path}  if its two endpoints belong to $X$ and $Y$, respectively. Particularly, if $X=\{x\}$ or $Y=\{y\}$, the path is referred to as an $(x, Y)$-path or $(X, y)$-path, respectively. Additionally, if $X=\{x\}$ and $|Y|=k$, we define an \textit{$(x, Y)$-fan} as a set of $k$ $(x, Y)$-paths that intersect only at the vertex $x$. Similarly, we can define a $(Y, x)$-fan.

The undirected toroidal mesh was proposed by Bhuyan and Agrawal~\cite{bhuyan1984generalized}, with recent related literature available in~\cite{juan2025mesh, niu2021two, pai2014queue}. The following is the formal definition: 

\begin{definition} \label{def:mesh}
{\rm (See~\cite{bhuyan1984generalized}.) For any $n$-tuple array $(d_1,d_2,\ldots,d_n)$, where $d_i \ge 2$ is an integer for $i \in \{1,2,\ldots,n\}$, the \textit{$n$-dimensional undirected toroidal mesh}, denoted as $C(d_1,d_2,\ldots,d_n)$, has the vertex set $V(C(d_1,d_2,\ldots,d_n))=\{u_1u_2 \cdots u_n \colon\, u_i \in \{0,1,\ldots,d_i-1\},\,d_i \ge 2,\,1 \le i \le n\}$, and two vertices $u_1u_2 \cdots u_n$ and $v_1v_2 \cdots v_n$ in $C(d_1,d_2,\ldots,d_n)$ are adjacent if and only if there exists an integer $j \in \{1,2,\ldots,n\}$ such that $\vert u_j - v_j \vert \equiv 1$ {\rm (mod} $d_j${\rm )} and $u_i = v_i$ for all $i \in \{1,2,\ldots,j-1,j+1,\ldots,n\}$.
}
\end{definition}

Fig.~\ref{fig:toroidal-mesh} illustrates the $2$-dimensional undirected toroidal mesh $C(3,2)$ and the $3$-dimensional undirected toroidal mesh $C(3,3,2)$.

\begin{figure}[ht]
	\centering
	\begin{subfigure}[c]{0.4\textwidth}
		\centering
		\includegraphics[width=0.75\textwidth]{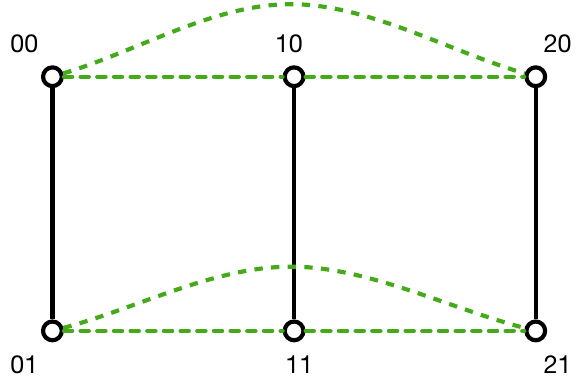}
		\caption{$C(3,2)$}
		\label{fig:C3_2}
	\end{subfigure}
	\begin{subfigure}[c]{0.4\textwidth}
		\centering
		\includegraphics[width=0.75\textwidth]{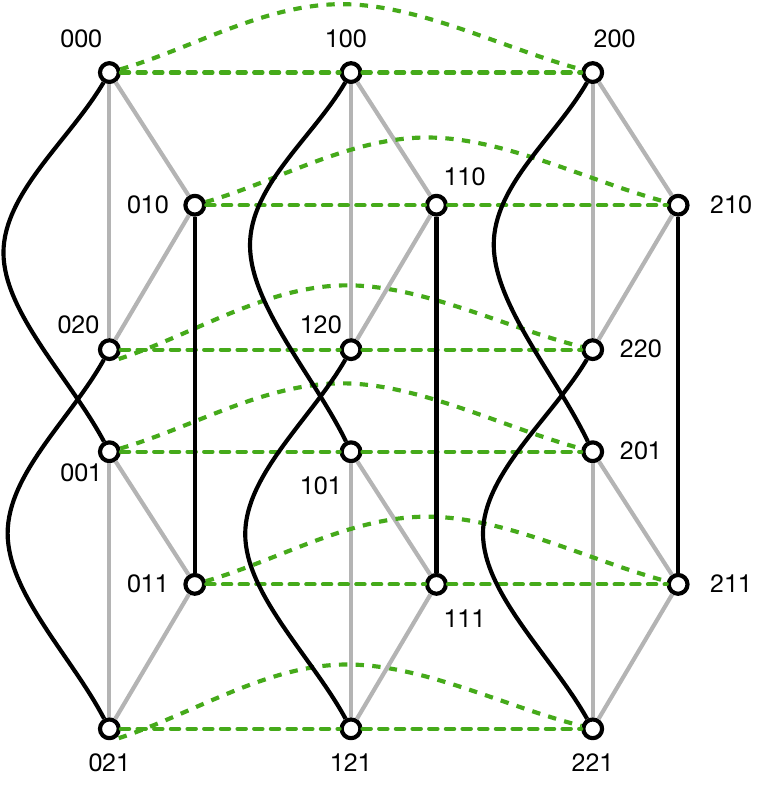}
		\caption{$C(3,3,2)$}
		\label{fig:C3_3_2}
	\end{subfigure}
	\caption{The illustration of undirected toroidal meshes}
	\label{fig:toroidal-mesh}
\end{figure}

The $n$-dimensional undirected toroidal mesh $C(d_1,d_2,\ldots,d_n)$ serves as a generalization of both the $k$-ary $n$-cube $Q_n^k$ and the $n$-dimensional hypercube $Q_n$. Specifically, when all dimensions are equal, meaning $d_1 = d_2 = \cdots = d_n = k$, the structure becomes $Q_n^k$. Additionally, when $d_1=d_2=\cdots=d_n=2$, the resulting structure is $Q_n$. Dvo{\v{r}}{\'a}k and Gu investigated the neighbor connectivity of both $Q_n$ and $Q_n^k$, as referenced in~\cite{dvovrak2020neighbor}. However, the neighbor connectivity of $C(d_1,d_2,\ldots,d_n)$ remains undetermined, and this will be addressed later.

Let's examine the structure of an $n$-dimensional undirected toroidal mesh. For $C(d_1,d_2,\ldots,d_n)$, where $n \ge 2$, by selecting a dimension $c$ (where $1 \le c \le n$), we can partition the mesh along the $c$-th dimension into $d_c$ subnetworks. For $0 \le i < d_c$, the $(i+1)$-th subnetwork, denoted as $C[i]$, is the subgraph of $C(d_1,d_2,\ldots,d_n)$ induced by the vertex set $\{u_1u_2 \cdots u_n \in V(C(d_1,d_2,\ldots,d_n)) \colon\, u_c=i\}$. It can be verified that $C[i]$ is isomorphic to $C(d_1,d_2,\ldots,d_{c-1},d_{c+1},\ldots,d_n)$. For example, when partitioning $C(3,3,2)$ along the second dimension, the subnetwork $C[0]$ is as shown in Fig.~\ref{fig:C0}. It can be observed from Fig.~\ref{fig:toroidal-mesh}(a) that $C[0]$ is isomorphic to $C(3,2)$.

\begin{figure}[ht]
	\centering
	\includegraphics[width=0.28\textwidth]{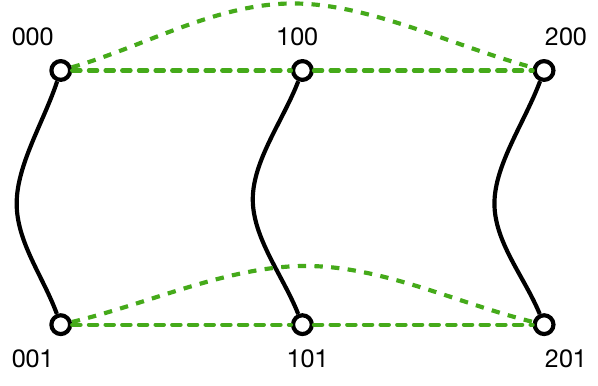}
	\caption{The illustration of $C[0]$ in $C(3,3,2)$}
	\label{fig:C0}
\end{figure}

Hereafter, for convenience, we assume that $C(d_1,d_2,\ldots,d_n)$ is partitioned into $d_c$ subnetworks, denoted as $C[0],C[1],\ldots,C[d_c-1]$, along the $c$-th dimension. Two different subnetworks $C[i]$ and $C[j]$ are said to be \textit{adjacent} if and only if $|i - j| \equiv 1$ (mod $d_c$). For $0 \le i < d_c$, a vertex $u=u_1u_2 \cdots u_{c-1}iu_{c+1} \cdots u_n$ in $C[i]$ is referred to as $u^i$. By Definition~\ref{def:mesh}, if a vertex $v\in V(C[j])$ with $j\ne i$ is adjacent to $u^i$ in the mesh, then $C[i]$ and $C[j]$ are adjacent. In this case, $v$ can be represented as $u^j=u_1u_2 \cdots u_{c-1}ju_{c+1} \cdots u_n$, which is called the \textit{outer neighbor} of $u$ in $C[j]$. Note that $u^j$ is the unique vertex adjacent to $u^i$ in $C[j]$. Obviously, for $d_k \ge 3$ (where $k \in \{1,2,\ldots,n\}$), a vertex $u$ in $C(d_1,d_2,\ldots,d_n)$ has exactly $2n$ neighbors. Specifically, these neighbors are:
\[
u_1u_2 \cdots u_{j-1}((u_j \pm 1)\ \text{mod}\ d_j)u_{j+1} \cdots u_n\ \text{for}\ j \in \{1,2,\ldots,n\}.
\]
Among these neighbors, there are $2n-2$ neighbors in $C[u_c]$. These are given by
\[
u_1u_2 \cdots u_{j-1}((u_j \pm 1)\ \text{mod}\ d_j)u_{j+1} \cdots u_n\ \text{for}\ j \in \{1,2,\ldots,c-1,c+1,\ldots,n\}.
\]
Additionally, two outer neighbors of $u$ are as follows: one is located in $C[(u_c+1)\ \text{mod}\ d_c]$, which is
\[
u_1u_2 \cdots u_{c-1}((u_c+1)\ \text{mod}\ d_c)u_{c+1} \cdots u_n,
\]
and the other is located in $C[(u_c-1)\ \text{mod}\ d_c]$, which is 
\[
u_1u_2 \cdots u_{c-1}((u_c-1)\ \text{mod}\ d_c)u_{c+1} \cdots u_n.
\]

From the above concepts, we have the following property.

\begin{lemma} \label{lm:neighbors}
Let $C=C(d_1,d_2,\ldots,d_n)$, where $n \ge 2$ and $d_i \ge 3$ for $i \in \{1,2,\ldots,n\}$, and let $u=u_1u_2 \cdots u_n \in V(C)$. If $C$ is partitioned along the $c$-th dimension {\rm (}$1 \le c \le n${\rm )}, then $u$ is located in $C[u_c]$ and it has exactly $2n$ neighbors in $C$, including $2n-2$ neighbors in $C[u_c]$.
\end{lemma}

Consider a simple example to illustrate Lemma~\ref{lm:neighbors}, as depicted in Fig.~\ref{fig:neighbors}. Let $u=00 \cdots 0$ be a vertex in $C(d_1,d_2,\ldots,d_n)$ and the mesh is partitioned along the first dimension. As we can see, $u$ has $2n-2$ neighbors in $C[0]$ and two out neighbors, one in $C[1]$ and another in $C[d_1-1]$.

\begin{figure}[ht]
	\centering
	\includegraphics[width=0.75\textwidth]{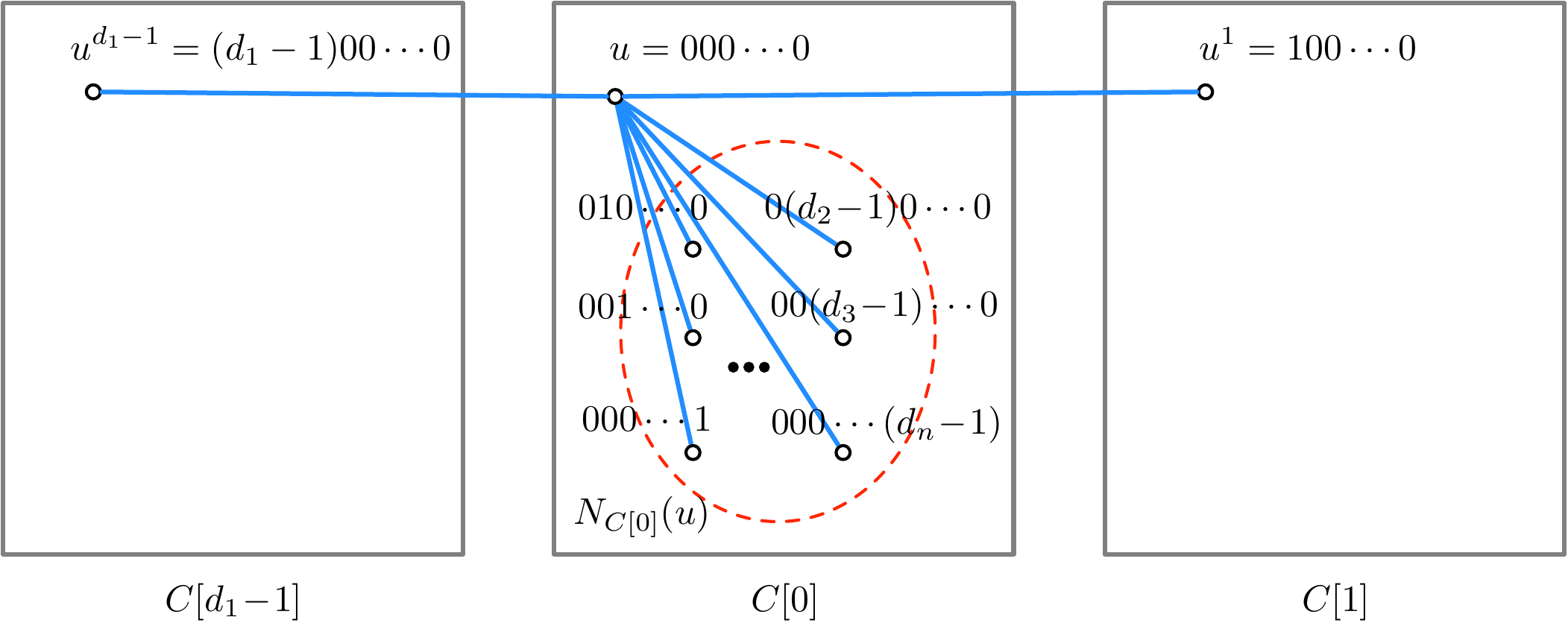}
	\caption{The distribution of neighbors for a vertex $u$ in $C(d_1,d_2,\ldots,d_n)$.}
	\label{fig:neighbors}
\end{figure}

Park~\cite{park2021torus} demonstrated the connectivity of $C(d_1,d_2,\ldots,d_n)$ to investigate the problem of disjoint path covers. The result is presented below.

\begin{lemma} \label{lm:connectivity}
{\rm (See~\cite{park2021torus}.)} 
Let $C=C(d_1,d_2,\ldots,d_n)$, where $n \ge 2$ and $d_i \ge 3$ for $1 \le i \le n$. Then, $\kappa(C)=2n$.
\end{lemma}

Dvo{\v{r}}{\'a}k and Gu~\cite{dvovrak2020neighbor} provided the following useful lemma.

\begin{lemma} \label{lm:disjoint-paths}
{\rm (See~\cite{dvovrak2020neighbor}.)} 
Let $f \ge 0$, $k \ge 1$ be integers and $G$ be an $(f+k)$-connected graph. If $F \subseteq V(G)$ with $\vert F \vert = f$ and $Y \subseteq V(G-F)$ with $\vert Y \vert = k$, then the following assertions hold.
\begin{enumerate}[{\rm(1)}]
\vspace{-0.3cm}
\item 
For any vertex set $X \subseteq V(G-F)$ with $\vert X \vert = k$, there is a set of $k$ pairwise disjoint $(X, Y)$-paths in $G-F$;
\vspace{-0.3cm}
\item 
For any vertex $x \in V(G-F)$, there is an $(x, Y)$-fan in $G-F$.
\end{enumerate}
\end{lemma}

\section{Neighbor connectivity of the undirected toroidal meshes} 
\label{Sec:main}

In this section, we first find the lower bound of the neighbor connectivity of the undirected toroidal mesh and then the upper bound. Hence, we can determine this network's neighbor connectivity.

\subsection{Lower bound of $\kappa_{NB}(C(d_1,d_2,\ldots,d_n))$}
\label{Sec:Lower-bound}

When examining the neighbor connectivity of a network, it is generally assumed that certain \textit{faulty source vertices} can affect neighboring vertices in the network. The faulty source vertices and the vertices that become faulty as a result are collectively referred to as \textit{faulty vertices}. All other vertices not falling into this category are considered \textit{healthy vertices}. Consequently, $U$ typically represents the set of all faulty source vertices, $N[U]$ denotes the set of all faulty vertices, and $G \ominus U$ refers to the subgraph generated by all remaining healthy vertices.

Before given the lower bound of $\kappa_{NB}(C(d_1,d_2,\ldots,d_n))$, we introduce some lemmas that will be used later.

\begin{lemma} \label{lm:common-neighbor}
Let $C=C(d_1,d_2,\ldots,d_n)$, where $n \ge 2$ and $d_i \ge 3$ for $1 \le i \le n$. If $x,y\in V(C)$ are two distinct vertices, then $\vert N(x) \cap N(y) \vert \in \{0,1,2\}$. Particularly, if $x$ and $y$ are adjacent, then $\vert N(x) \cap N(y) \vert \in \{0,1\}$.
\end{lemma}

\begin{proof}
Since $x$ and $y$ are two distinct vertices, we can partition $C(d_1,d_2,\ldots,d_n)$ along the $c$-th dimension, where $1 \le c \le n$, such that $x \in V(C[i])$ and $y \in V(C[j])$ for $0 \le i,j < d_c$ and $i \neq j$. Note that the subgraph of $C(d_1,d_2,\ldots,d_n)$ induced by the vertex set $\{x^0,x^1,\ldots,x^{d_c-1}\}$ forms a cycle, denoted as $C_1$. Similarly, the induced subgraph of the vertex set $\{y^0,y^1,\ldots,y^{d_c-1}\}$ in $C(d_1,d_2,\ldots,d_n)$ is also a cycle, denoted as $C_2$. Consider the following two cases.

\medskip
\textbf{Case 1:} $V(C_1) \cap V(C_2) \neq \varnothing$.

In this case, $x$ and $y$ are in the same cycle (i.e., $C_1 = C_2$), with $x^r=y^r$ for $r \in \{0,1,\ldots,d_c-1\}$. We discuss two scenarios based on whether $x$ and $y$ are adjacent.

\medskip
\textbf{Case 1.1:} $x$ and $y$ are adjacent.

Since $x \in V(C[i])$ and $y \in V(C[j])$ are two adjacent vertices, by Lemma \ref{lm:neighbors}, $C[i]$ and $C[j]$ are adjacent subnetworks. Moreover, $y$ is the unique outer neighbor of $x$ in $C[j]$, and $x$ is the unique outer neighbor of $y$ in $C[i]$. If $d_c=3$, then $C[i]$, $C[j]$, and the remaining subnetwork are pairwise adjacent. Thus, $x$ and $y$ have a unique common outer neighbor, say $u$, in the third subnetwork, i.e., $|N(x) \cap N(y)| = |\{u\}| = 1$ (see Fig.~\ref{fig:cn_case1-1}(a)). One the other hand, if $d_c \ge 4$, then $x$ and $y$ have no common neighbors. Thus, $\vert N(x) \cap N(y) \vert = 0$ (see Fig.~\ref{fig:cn_case1-1}(b)).

\begin{figure}[ht]
	\centering
	\begin{subfigure}{0.8\textwidth}
		\centering
		\includegraphics[width=0.66\textwidth]{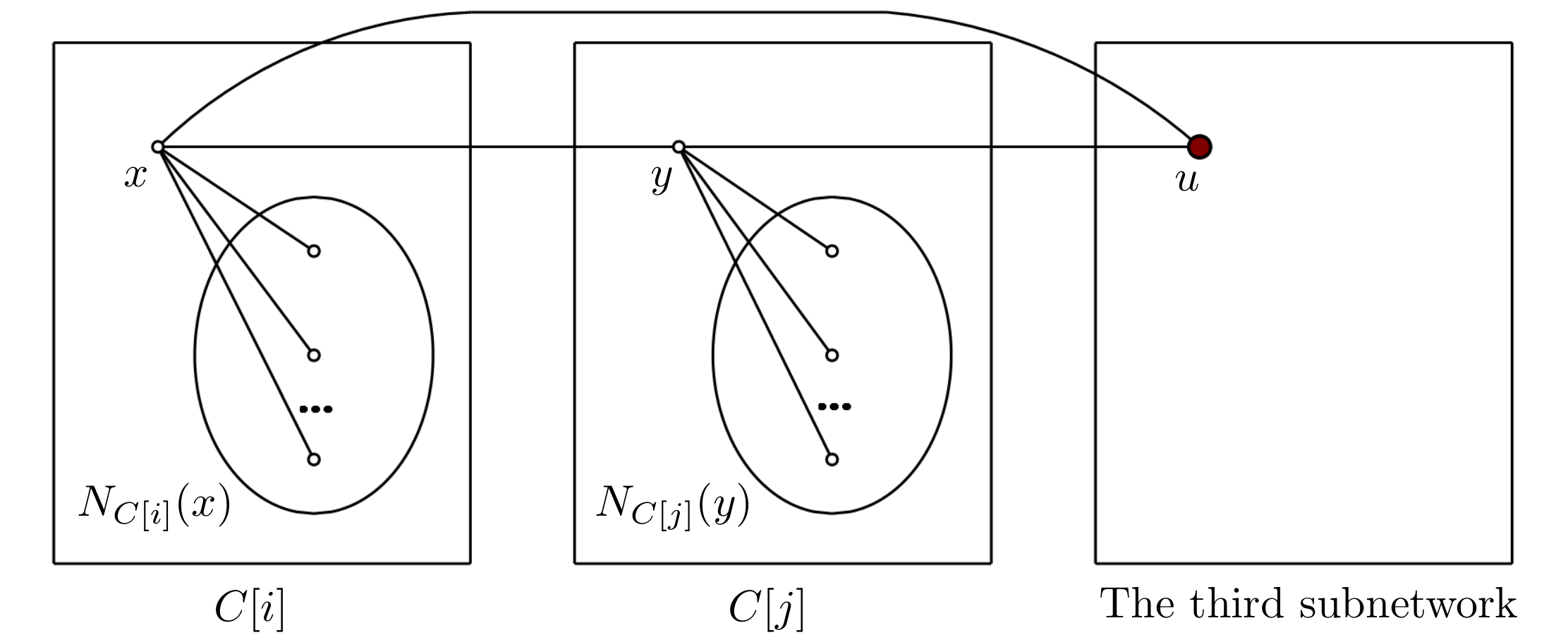}
		\caption{The situation when $d_c=3$}
		\label{fig:cn_case1-1_a}
	\end{subfigure}
	\vspace{5pt}

	\begin{subfigure}{0.8\textwidth}
		\centering
		\includegraphics[width=0.90\textwidth]{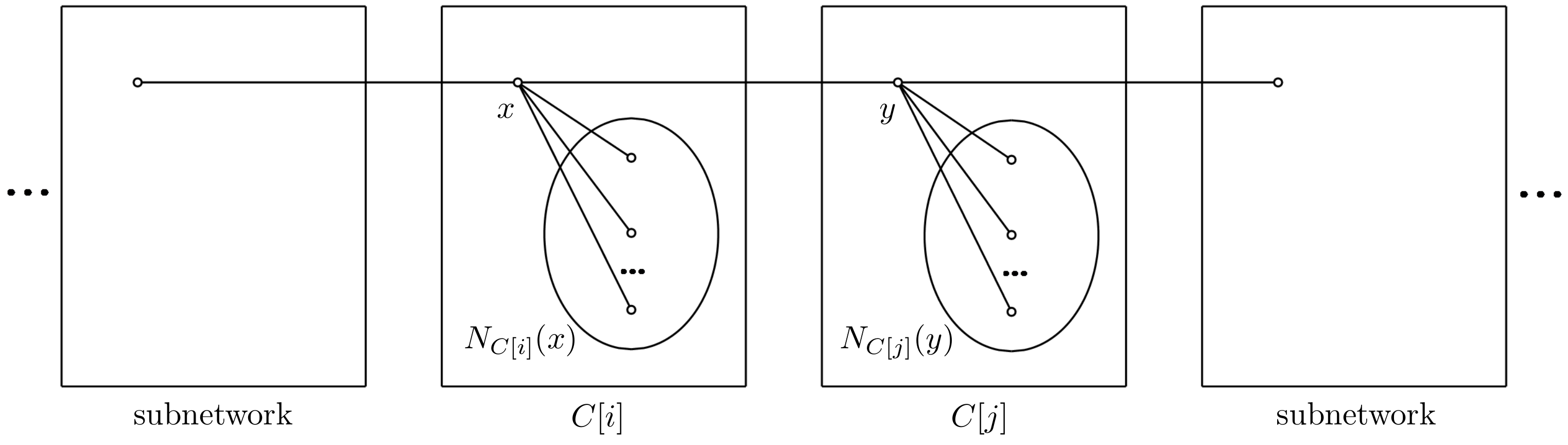}
		\caption{The situation when $d_c \ge 4$}
		\label{fig:cn_case1-1_b}
	\end{subfigure}
	\caption{The illustration of Case 1.1 in Lemma~\ref{lm:common-neighbor}.}
	\label{fig:cn_case1-1}
\end{figure}

\medskip
\textbf{Case 1.2:} $x$ and $y$ are not adjacent.

Since $x$ and $y$ are in the same cycle, we have $d_c \ge 4$ by Lemma \ref{lm:neighbors}.

\begin{figure}[H]
	\centering
	\begin{subfigure}{0.8\textwidth}
		\centering
		\includegraphics[width=0.75\textwidth]{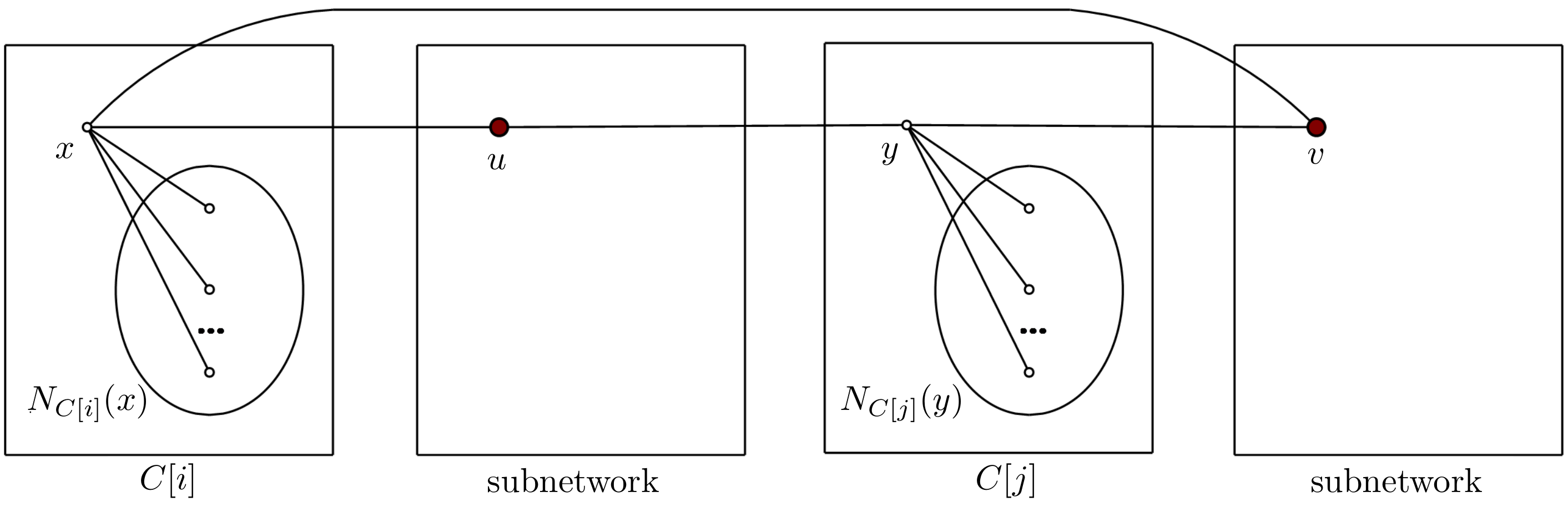}
		\caption{The situation when $d_c=4$}
		\label{fig:cn_case1-2_a}
	\end{subfigure}
	\vspace{5pt}

	\begin{subfigure}{0.8\textwidth}
		\centering
		\includegraphics[width=0.94\textwidth]{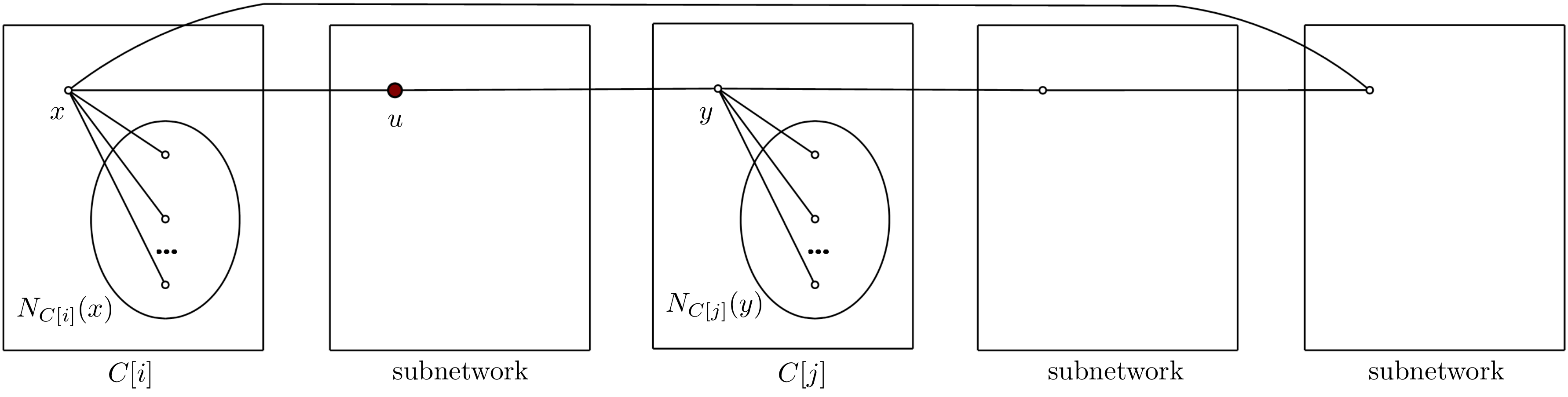}
		\caption{The situation when $d_c=5$}
		\label{fig:cn_case1-2_b}
	\end{subfigure}
	\vspace{5pt}

	\begin{subfigure}{0.8\textwidth}
		\centering
		\includegraphics[width=1.00\textwidth]{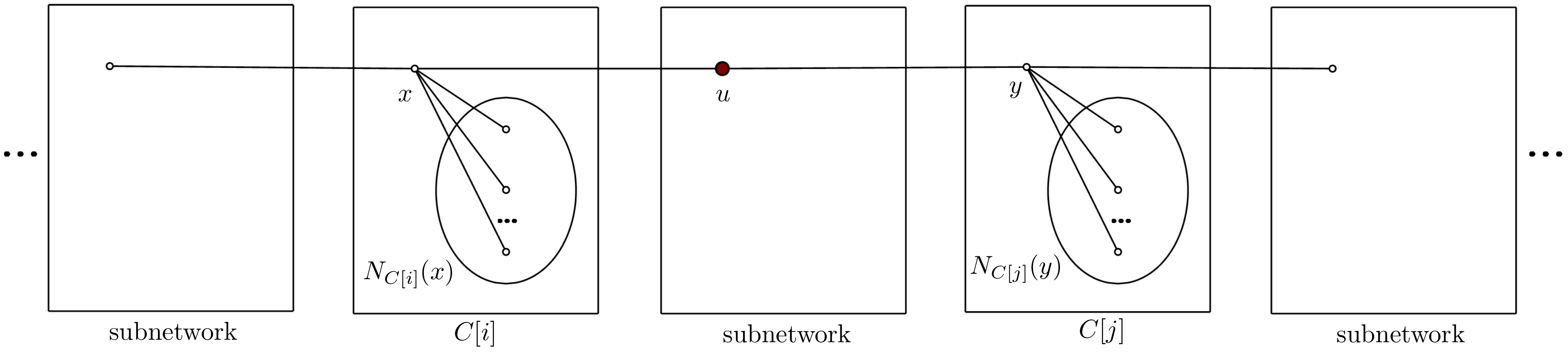}
		\caption{The situation when $d_c \ge 6$ and $|N(x) \cap N(y)| = 1$}
		\label{fig:cn_case1-2_c}
	\end{subfigure}
	\vspace{5pt}

	\begin{subfigure}{0.8\textwidth}
		\centering
		\includegraphics[width=1.00\textwidth]{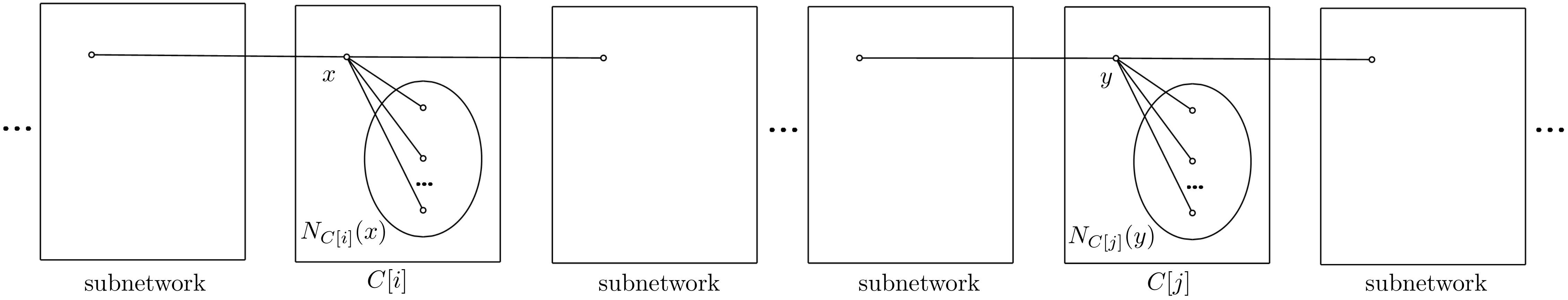}
		\caption{The situation when $d_c \ge 6$ and $|N(x) \cap N(y)| = 0$}
		\label{fig:cn_case1-2_d}
	\end{subfigure}
	\caption{The illustration of Case 1.2 in Lemma~\ref{lm:common-neighbor}.}
	\label{fig:cn_case1-2}
\end{figure}

If $d_c = 4$, then $x$ and $y$ have exactly two common outer neighbors, say $u$ and $v$. Thus, $|N(x) \cap N(y)| = |\{u, v\}| = 2$ (see Fig.~\ref{fig:cn_case1-2}(a)).

If $d_c = 5$, then $x$ and $y$ have only one common outer neighbor, say $u$. Thus, $|N(x) \cap N(y)| = |\{u\}| = 1$ (see Fig.~\ref{fig:cn_case1-2}(b)).

When $d_c \ge 6$, two situations need to be considered. If $C[i]$ and $C[j]$ share a common adjacent subnetwork, this subnetwork exists uniquely. In this case, $x$ and $y$ have a unique common neighbor, say $u$, within this subnetwork. Thus, $|N(x) \cap N(y)| = |\{u\}| = 1$ (see Fig.~\ref{fig:cn_case1-2}(c)). Alternatively, if they do not share a common adjacent subnetwork, then $|N(x) \cap N(y)| = 0$ (see Fig.~\ref{fig:cn_case1-2}(d)).

\textbf{Case 2:} $V(C_1) \cap V(C_2) = \varnothing$.

In this case, $x$ and $y$ are in different cycles, with $x^r \neq y^{r'}$ for $r,r' \in \{0,1,\ldots,d_c-1\}$. By Lemma \ref{lm:neighbors}, $x$ and $y$ are not adjacent. We discuss two scenarios based on whether $x$ and $y$ are located in adjacent subnetworks.

\medskip
\textbf{Case 2.1:} $C[i]$ and $C[j]$ are non-adjacent subnetworks.

In this case, $d_c \ge 4$, and by Lemma \ref{lm:neighbors}, $\vert N(x) \cap N(y) \vert = 0$. This situation is illustrated in Fig~\ref{fig:cn_case2}(a). 

\medskip
\textbf{Case 2.2:} $C[i]$ and $C[j]$ are adjacent subnetworks.

Without loss of generality, let $i=0$ and $j=1$. Hence, we have $x=x^0 \in V(C[0])$ and $y=y^1 \in V(C[1])$. If $(x^1, y) \in E(C[1])$, then it follows that $(y^0, x) \in E(C[0])$. Conversely, if $(y^0, x) \in E(C[0])$, then $(x^1, y) \in E(C[1])$. In this situation, $N(x) \cap N(y) = \{x^1, y^0\}$. On the other hand, if $(x^1, y) \notin E(C[1])$, then $(y^0, x) \notin E(C[0])$. Similarly, if $(y^0, x) \notin E(C[0])$, then $(x^1, y) \notin E(C[1])$. In this situation, we can conclude that $N(x) \cap N(y) = \varnothing$. This shows that $\vert N(x) \cap N(y) \vert \in \{0,2\}$. These scenarios are illustrated in Fig.~\ref{fig:cn_case2}(b) and Fig.~\ref{fig:cn_case2}(c), respectively.
\end{proof}

\begin{figure}[htbp]
	\centering
	\begin{subfigure}{0.85\textwidth}
		\centering
		\includegraphics[width=1.00\textwidth]{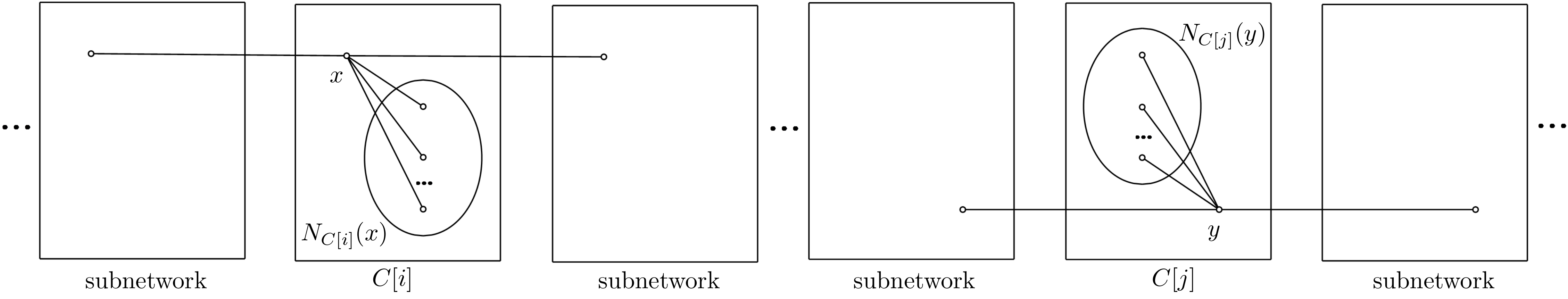}
		\caption{The illustration of Case 2.1}
		\label{fig:cn_case2_a}
	\end{subfigure}
	\vspace{5pt}

	\begin{subfigure}{0.85\textwidth}
		\centering
		\includegraphics[width=0.75\textwidth]{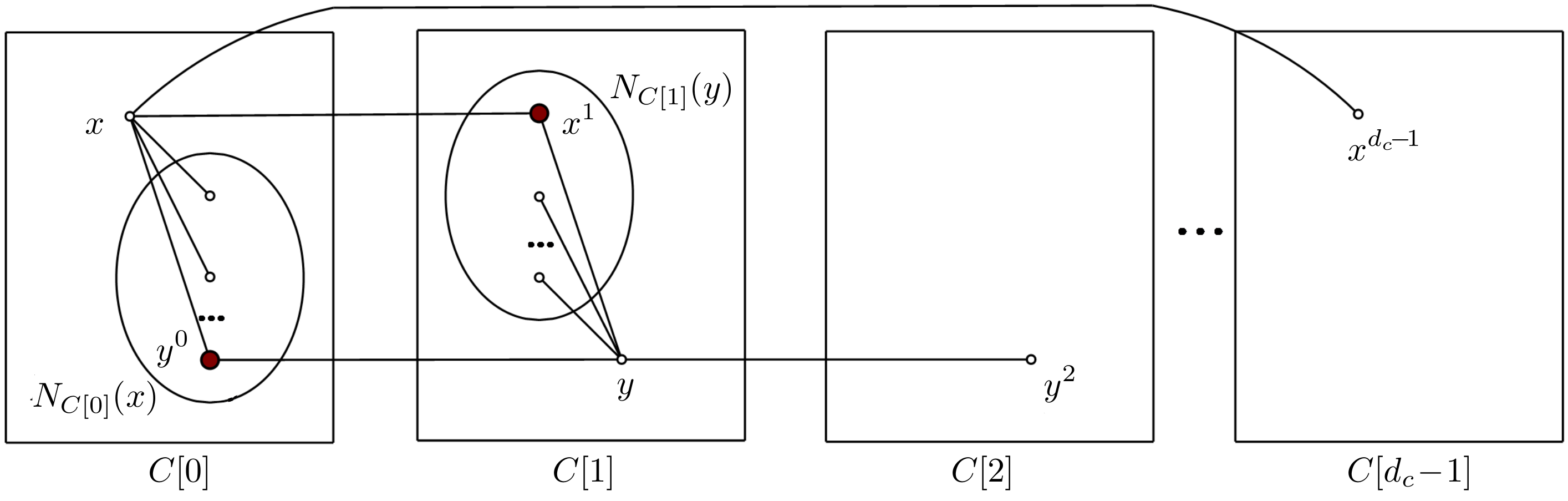}
		\caption{The situation of Case 2.2 with $\vert N(x) \cap N(y) \vert = 2$}
		\label{fig:cn_case2_b}
	\end{subfigure}
	\vspace{5pt}

	\begin{subfigure}{0.85\textwidth}
		\centering
		\includegraphics[width=0.75\textwidth]{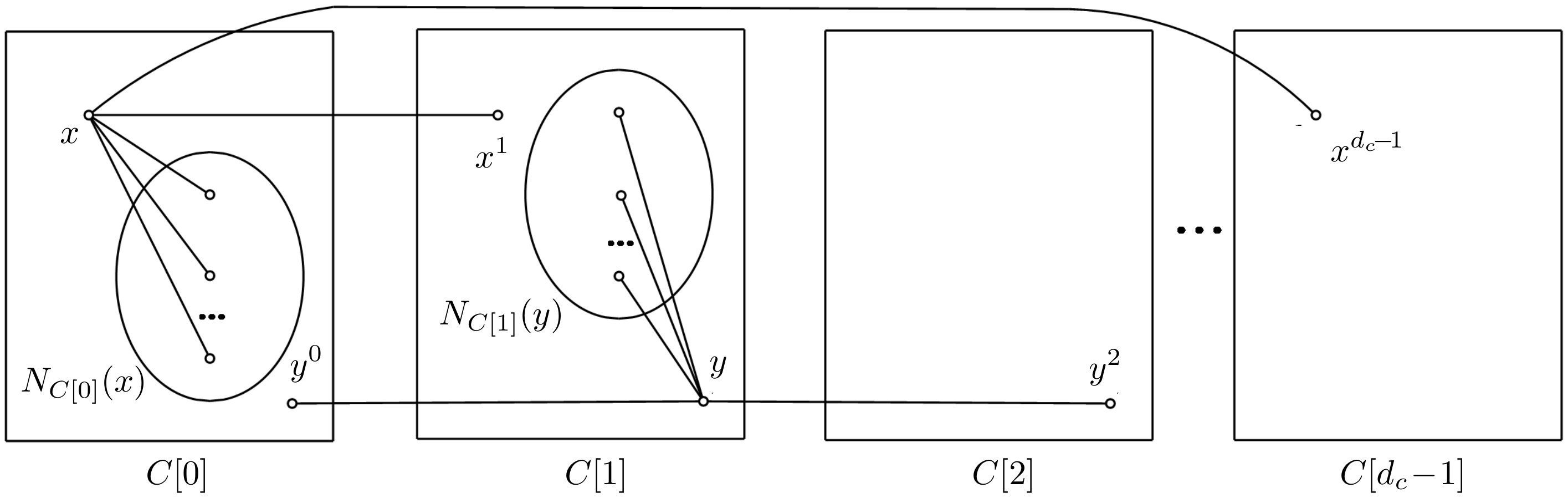}
		\caption{The situation of Case 2.2 with $\vert N(x) \cap N(y) \vert = 0$}
		\label{fig:cn_case2_c}
	\end{subfigure}
	\caption{The illustration of Case 2 in Lemma~\ref{lm:common-neighbor}.}
	\label{fig:cn_case2}
\end{figure}

\begin{lemma} \label{lm:healthy-vertex}
Let $C=C(d_1,d_2,\ldots,d_n)$, where $n \ge 3$ and $d_i \ge 3$ for $1 \le i \le n$, such that it is partitioned into $d_c$ subnetworks along a dimension $c$ where $1 \le c\le n$, and let $U\subset V(C)$ be a set of faulty source vertices. If $|U|=\ell < n$, then $C[i]$ contains at least one healthy vertex for any $0 \le i < d_c$.
\end{lemma}

\begin{proof}
For each $i\in\{0,1,\ldots, d_c-1\}$, let $U_i = U \cap V(C[i])$ and $\mu_i = \vert U_i \vert$. Hereafter, for notational convenience, the indices of the terms $U_i$ and $\mu_i$ are considered modulo $d_c$. Let
\[
F_{i1} = U_i,\ F_{i2} = N_{C[i]}(U_i),\ \text{and}\ F_{i3} = N_{C[i]}(U_{i+1} \cup U_{i-1}).
\]
Also, let $f_{i1}=\vert F_{i1} \vert$, $f_{i2}=\vert F_{i2} \vert$, and $f_{i3}=\vert F_{i3} \vert$ be the cardinalities of these three sets, respectively. Clearly, $F_{i1} \cup F_{i2} \cup F_{i3}$ contains all faulty vertices in $C[i]$.

Firstly, we know that $f_{i1}=\mu_i$. By Lemma \ref{lm:neighbors}, we have
\[
f_{i2} = |N_{C[i]}(U_i)| \le (2n-2)\mu_i
\]
and
\[
f_{i3} = |N_{C[i]}(U_{(i+1)} \cup U_{(i-1)})| \le \mu_{i-1} + \mu_{i+1} \le \ell-\mu_i.
\]
Therefore, the total number of faulty vertices in $C[i]$ is at most
\[
f_{i1}+f_{i2}+f_{i3} \le \mu_i + (2n-2)\mu_i + (\ell-\mu_i) = (2n-2)\mu_i+\ell.
\]
Since $|V(C[i])|=\prod_{j=1,\, j \neq c}^n d_j$, where $d_j \ge 3$, and $\mu_i \le \ell < n$,  the number of healthy vertices in $C[i]$ is at least
\[
\bigg(\prod_{\substack{j\in\{1,2,\ldots,n\}\\ j\ne c}} d_j\bigg) - (2n - 2)\mu_i  - \ell \geq
\begin{cases}
3^{n-1} - (2n - 1)\ell \geq 3^{n-1} - (2n - 1)(n - 1) \geq 6\ \text{when}\ n \geq 4; \\
3^{n-1} - (2n - 1)\ell \geq 4\ \text{when}\ n = 3\ \text{and}\ \ell \leq 1; \\
3^{n-1} - (2n - 2) - 2 = 3\ \text{when}\ \mu_i \leq 1 < 2 = \ell < n = 3.
\end{cases}
\]
This indicates that there is at least one healthy vertex in $C[i]$ in the above three cases. 

It remains to consider the case $\mu_i=\ell=2<n=3$. In this case, we have $f_{i3}=0$ and $|V(C[i])|=\prod_{j=1,\, j \neq c}^n d_j \ge 3^{n-1} = 9$. If $|V(C[i])| \geq 11$, the number of healthy vertices in $C[i]$ is at least
\[
|V(C[i])| - (2n - 2)\mu_i - \ell \geq 11 - (2n - 2)2 - 2 = 1.
\]
If $9 \leq |V(C[i])| \leq 10$, we note that the total number of vertices in  $C[i]$ can only be 9 because $d_j \geq 3$ is an integer for all $1 \leq j \leq n$ and $j \neq c$. Consequently, the possibility of $|V(C[i])| = 10$ is eliminated. Since $f_{i1}=\mu_i=2$, let $F_{i1} = U_i = \{x,y\}$. Given that $|V(C[i])| = 9$, the subnetwork $C[i]$ is isomorphic to $C(3,3)$. It is easy to observe that if $x$ and $y$ are adjacent in $C[i]$, then $|N(x) \cap N(y)| = 1$. In contrast, if $x$ and $y$ are not adjacent, then $|N(x) \cap N(y)| = 2$. By Lemma~\ref{lm:neighbors}, the former case indicates that $f_{i2} = |N_{C[i]}(U_i)| = |N_{C[i]}(x)|+|N_{C[i]}(y)|-|N(x) \cap N(y)|-|U_i| = 2(2n-2)-1-2=5$, and the latter case implies $f_{i2}= |N_{C[i]}(x)|+|N_{C[i]}(y)|-|N(x) \cap N(y)| = 2(2n-2)-2=6$. Thus, $f_{i2} \leq \max\{5, 6\} = 6$. In summary, the number of faulty vertices in $C[i]$ is at most $f_{i1} + f_{i2} + f_{i3} \leq 2 + 6 + 0 = 8$. It follows that the number of healthy vertices in $C[i]$ is at least $|V(C[i])|-8=9-8=1$.
\end{proof}

\medskip
\begin{lemma} \label{lm:healthy-pairs}
Let $C=C(d_1,d_2,\ldots,d_n)$, where $n \ge 3$ and $d_k \ge 3$ for $1 \le k \le n$, and let $U\subset V(C)$ be a set of faulty source vertices with $|U|=\ell<n$. If $C[i]$ and $C[j]$ are two adjacent subnetworks in $C$ such that $\mu_i = |U \cap V(C[i])|$ and $\mu_j = |U \cap V(C[j])|$, then
\[ 
| \{ v \in V(C[i]) \colon\, \text{both}\ v\ \text{and}\ v^j\ \text{are healthy vertices} \} | > h, 
\]
where $h = 2n - 2 - \ell - \mu_i - \mu_j$. Specifically, if $x\in V(C[i])$ is a healthy vertex, then at least $h$ of these vertices are adjacent to $x$ in $C[i]$.
\end{lemma}

\begin{proof}
Let $U_i = U \cap V(C[i])$ and $U_j = U \cap V(C[j])$. Then, $\mu_i = |U_i|$ and $\mu_j = |U_j|$. Furthermore, let 
\[
H = \{ v \in V(C[i]) \colon\, v \notin N[U] \text{ and } v^j \notin N[U] \}. 
\]

By Lemma~\ref{lm:healthy-vertex}, let $x\in V(C[i])$ be a healthy vertex. We assert that $H$ contains at least $h = 2n - 2 - \ell - \mu_i - \mu_j$ healthy neighbors of $x$ in $C[i]$. By Lemma~\ref{lm:neighbors}, $x$ has $2n - 2$ neighbors in $C[i]$. Since $x$ is a healthy vertex, we have $N(x) \cap U_i = \varnothing$. By Lemma~\ref{lm:common-neighbor}, $|N(x) \cap N(U_i)| \le 2\mu_i$. Thus, $|N(x) \cap N[U_i]| = |N(x) \cap U_i|+|N(x) \cap N(U_i)| \le 2\mu_i$, which follows that at least $(2n - 2) - 2\mu_i$ neighbors of $x$ do not belong to $N_{C[i]}[U_i]$.

If $x^j$ is a healthy vertex, a similar argument as above demonstrates that $|N(x^j) \cap N[U_j]| \le 2\mu_j$. In contrast, if $x^j$ is not a healthy vertex, then $x^j \notin U_j \cup N(U_i)$ since $x$ is a healthy vertex. This further implies that $x^j \in N(U_j)$ or $x^j \in N(U \setminus (U_i \cup U_j))$. If $x^j \notin N(U_j)$, then $N(x^j) \cap U_j = \varnothing$. By Lemma \ref{lm:common-neighbor}, $|N(x^j) \cap N(U_j)| \le 2\mu_j$. Thus, $|N(x^j) \cap N[U_j]| = |N(x^j) \cap U_j|+|N(x^j) \cap N(U_j)| \le 2\mu_j$. If $x^j \in N(U_j)$, let $p=|N(x^j) \cap U_j|$, where $1 \le p \le \mu_j$. Since $x^j$ has $p$ neighbors in $U_j$, it follows that there are $\mu_j - p$ faulty source vertices of $U_j$ that are not neighbors of $x^j$. By Lemma \ref{lm:common-neighbor}, these $p$ neighbors of $x^j$ have at most $p$ common neighbors with $x^j$, and those $\mu_j - p$ faulty source vertices can share at most $2(\mu_j - p)$ common neighbors with $x^j$. Thus,
\[
|N(x^j) \cap N(U_j)| \le p+2(\mu_j-p) = 2\mu_j-p.
\]
As a consequence, we have
\[
|N(x^j) \cap N[U_j]| = |N(x^j) \cap U_j| + |N(x^j) \cap N(U_j)| \le p + (2\mu_j-p) = 2\mu_j.
\]
This situation is illustrated in Fig.~\ref{fig:healthy_vertex_pairs-1}. In summary, whether $x^j$ is a healthy vertex or not, we have
\[
|N(x^j) \cap N[U_j]| \le 2\mu_j.
\]

\begin{figure}[ht]
	\centering
	\includegraphics[width=0.42\textwidth]{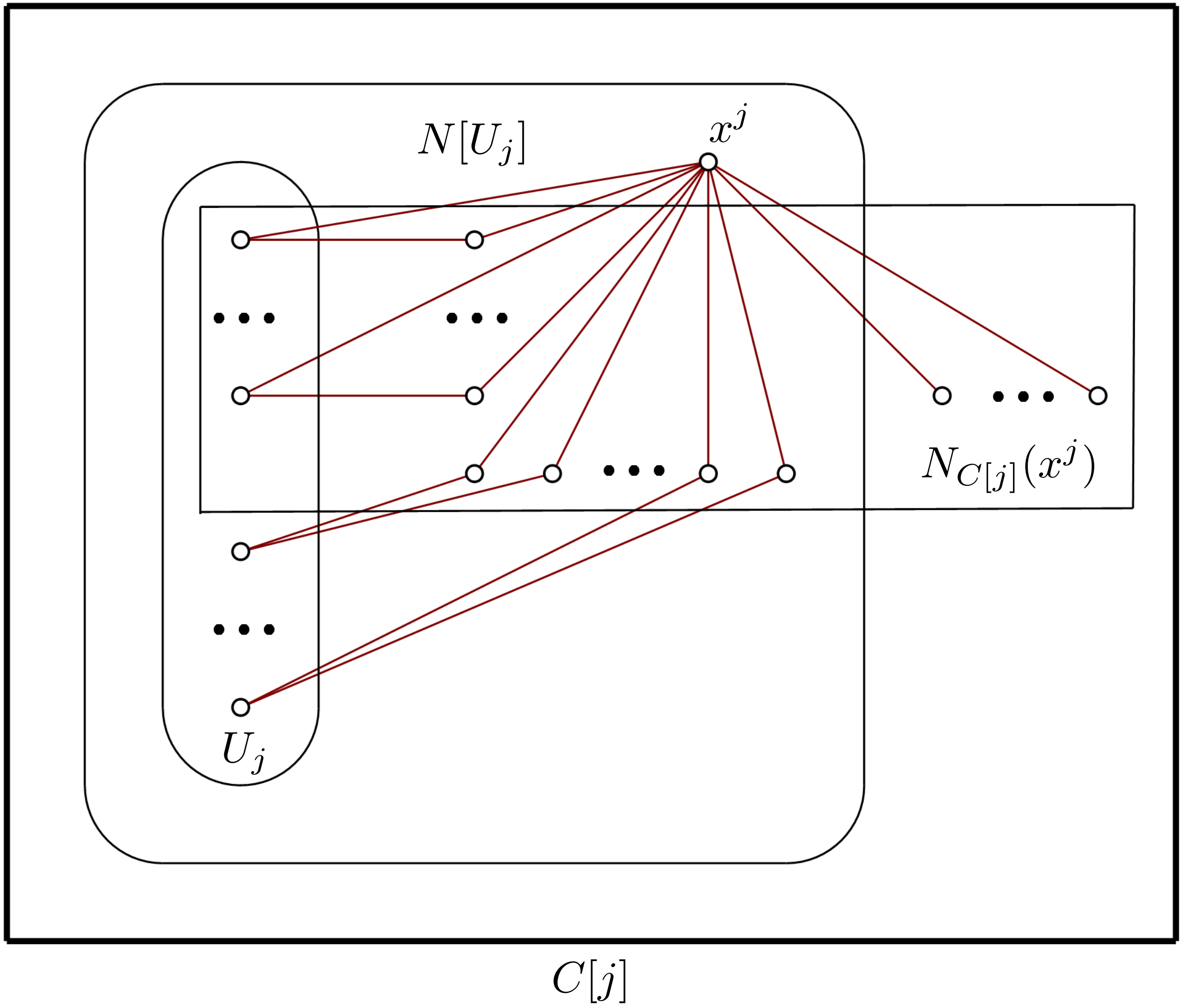}
	\caption{The illustration for $x^j \in N(U_j)$ in Lemma~\ref{lm:healthy-pairs}.}
	\label{fig:healthy_vertex_pairs-1}
\end{figure}

Based on the above discussion, $x$ has at least $(2n - 2) - 2\mu_i - 2\mu_j$ neighbors in $C[i]$, denoted as $x_1, x_2, \ldots, x_{2(n-\mu_i-\mu_j-1)}$, that satisfy the following condition:
\[
x_k \notin N[U_i]\ \text{and}\ x_k^j \notin N[U_j]\ \text{for}\ k\in\{1,2,\ldots,2(n-\mu_i-\mu_j-1)\}.
\]
If $x_k \in N[U_j]$, it implies $x_k^j \in U_j$, contradicting $x_k^j \notin N[U_j]$. Similarly, if $x_k^j \in N[U_i]$, it implies $x_k \in U_i$, contradicting $x_k \notin N[U_i]$. Hence, we have 
\[
x_k \notin N[U_j]\ \text{and}\ x_k^j \notin N[U_i]\ \text{for}\ k\in\{1,2,\ldots,2(n-\mu_i-\mu_j-1)\}.
\]
Thus, we can conclude that
\[
x_k , x_k^j \notin N[U_i] \cup N[U_j]\ \text{for}\ k\in\{1,2,\ldots,2(n-\mu_i-\mu_j-1)\}.
\]

Finally, we need to confirm whether $x_k$ and $x_k^j$ belong to $N(U \setminus (U_i \cup U_j))$. If $x_k \in N(U \setminus (U_i \cup U_j))$, it reflects the presence of a faulty source neighbor of $x_k$ in a subnetwork adjacent to $C[i]$ and different from $C[j]$. Similarly, if $x_k^j \in N(U \setminus (U_i \cup U_j))$, it indicates that there is a faulty source neighbor of $x_k^j$ in a subnetwork adjacent to $C[j]$, which differs from $C[i]$. Clearly, the number of faulty source vertices in these two subnetworks cannot exceed $\ell - \mu_i - \mu_j$. Consequently, we can conclude that $x$ has at least $(2n - 2) - 2\mu_i - 2\mu_j - (\ell - \mu_i - \mu_j) = 2n - 2 - \ell - \mu_i - \mu_j = h$ healthy neighbors in $C[i]$. These neighbors are denoted as $x_1^{\prime}, x_2^{\prime}, \ldots, x_h^{\prime}$, and they satisfy the following condition:
\[
x_k^{\prime}, (x_k^{\prime})^j \notin N[U]\ \text{for}\ k \in \{1,2,\ldots,h\}.
\]
Thus, it is confirmed that $H$ indeed contains at least $h$ healthy neighbors of $x$ in $C[i]$.

\begin{figure}[H]
	\centering
	\begin{subfigure}{0.48\textwidth}
		\centering
		\includegraphics[width=0.9\textwidth]{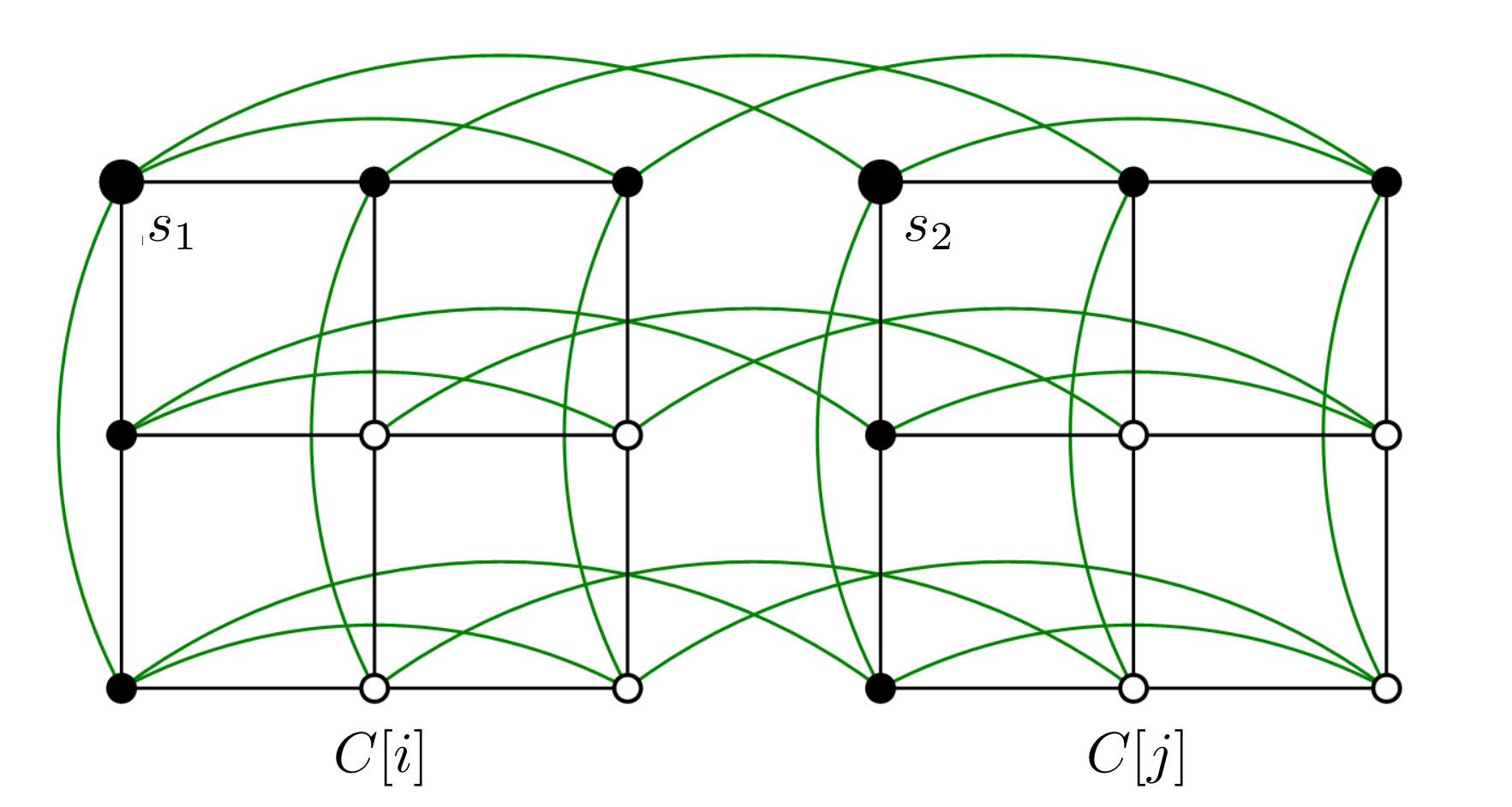}
		\caption{$s_1\in V(C[i])$ and $s_2\in V(C[j])$ are adjacent.}
		\label{fig:healthy_vertex_pairs-2a}
	\end{subfigure}
	\begin{subfigure}{0.48\textwidth}
		\centering
		\includegraphics[width=0.9\textwidth]{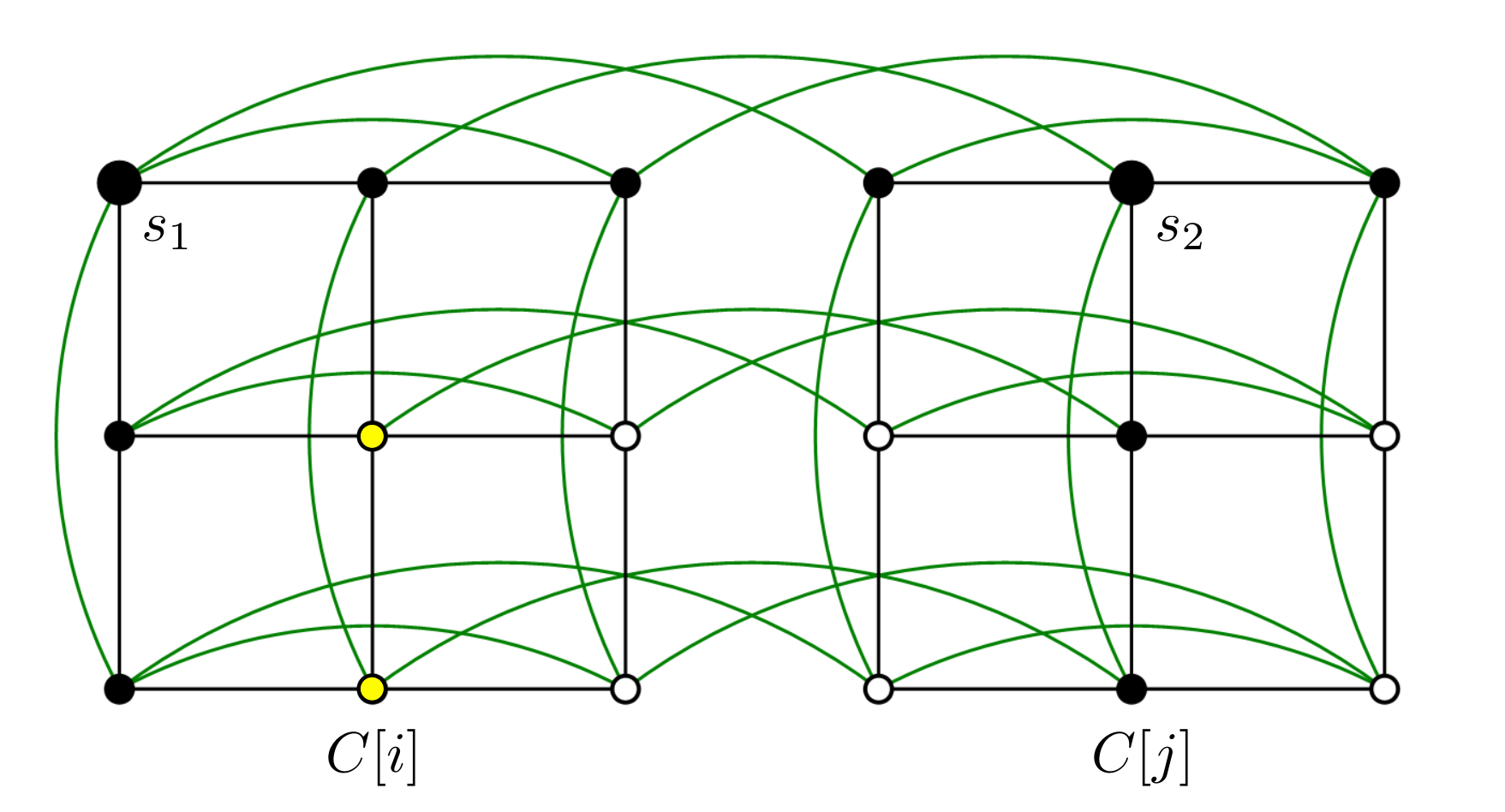}
		\caption{$s_1\in V(C[i])$ and $s_2\in V(C[j])$ are not adjacent.}
		\label{fig:healthy_vertex_pairs-2b}
	\end{subfigure}
	\vspace{15pt}

	\begin{subfigure}{0.48\textwidth}
		\centering
		\includegraphics[width=1.00\textwidth]{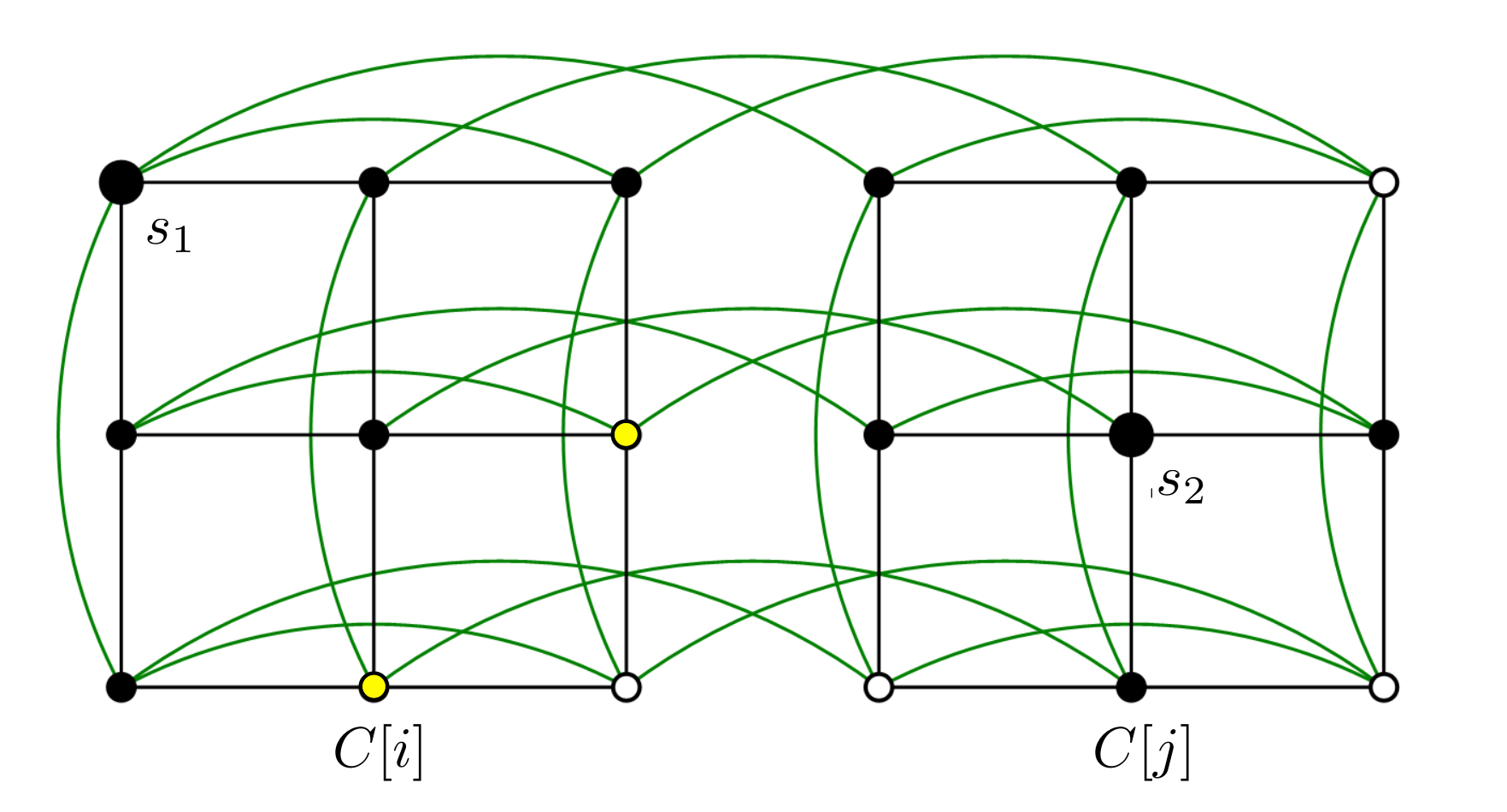}
		\caption{$s_1\in V(C[i])$ and $s_2\in V(C[j])$ are not adjacent.}
		\label{fig:healthy_vertex_pairs-2c}
	\end{subfigure}
	\begin{subfigure}{0.48\textwidth}
		\centering
		\includegraphics[width=1.00\textwidth]{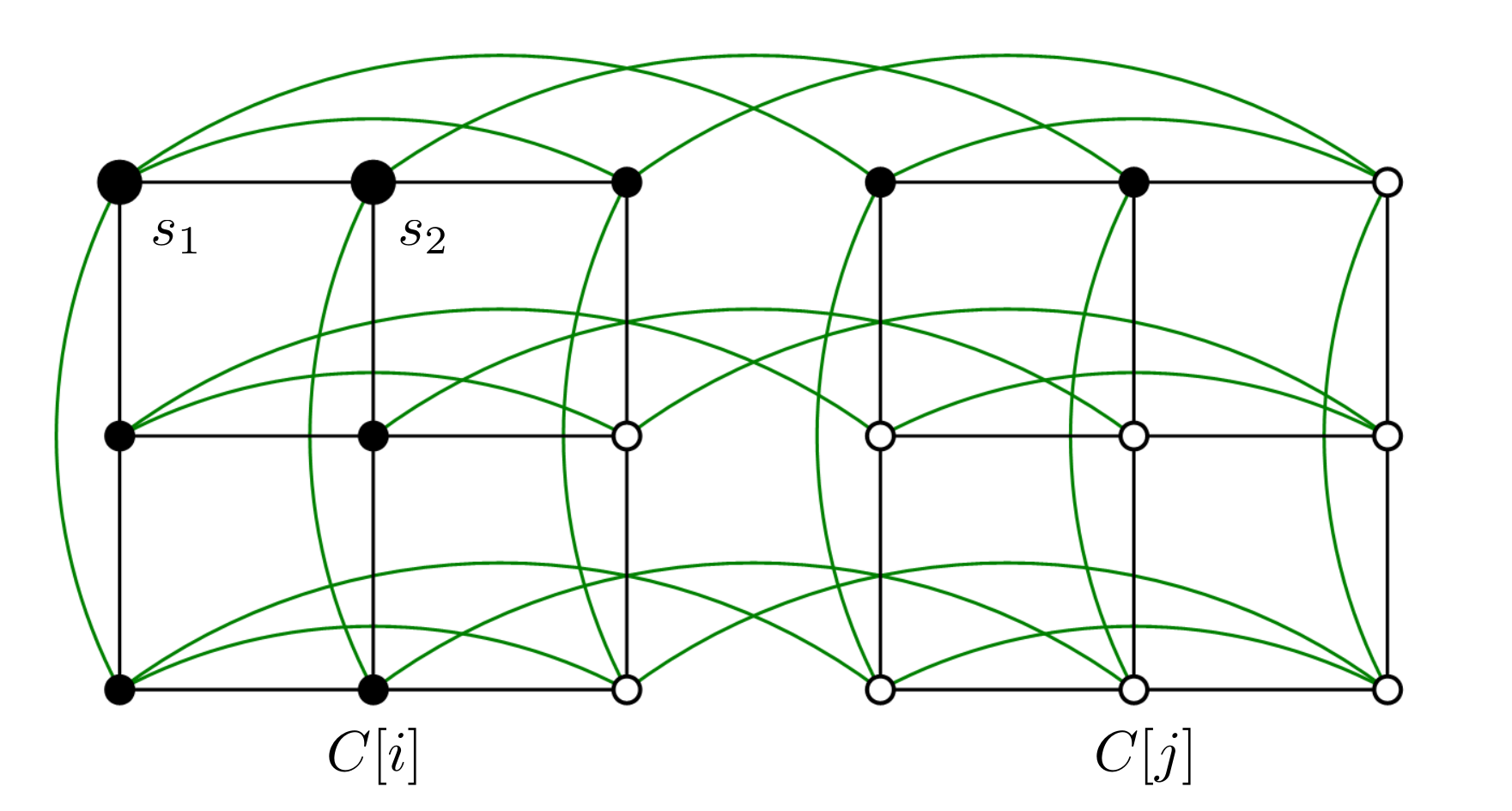}
		\caption{$s_1,s_2\in V(C[i])$ are adjacent.}
		\label{fig:healthy_vertex_pairs-2d}
	\end{subfigure}
	\vspace{15pt}

	\begin{subfigure}{0.48\textwidth}
		\centering
		\includegraphics[width=1.00\textwidth]{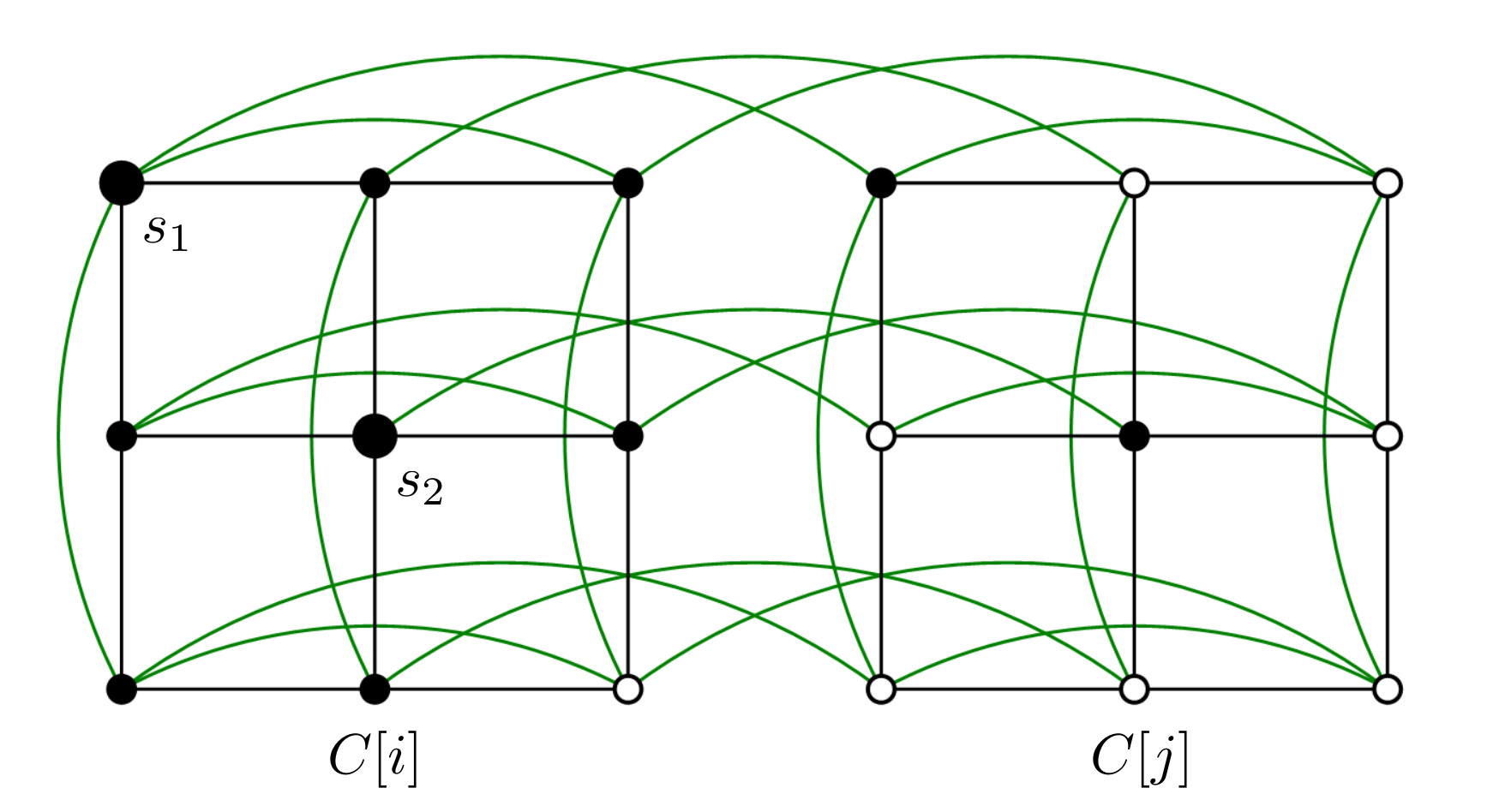}
		\caption{$s_1,s_2\in V(C[i])$ are not adjacent.}
		\label{fig:healthy_vertex_pairs-2e}
	\end{subfigure}
	\caption{The illustration for $n=3$, $|V(C[i])|=9$, and $\mu_i+\mu_j=2$ in Lemma~\ref{lm:healthy-pairs}, showing black nodes as faulty vertices, white nodes as healthy vertices, and yellow nodes as healthy vertices in $C[i]$ with corresponding outer neighbors being faulty in $C[j]$.}
	\label{fig:healthy_vertex_pairs-2}
\end{figure}

Below, we prove that the number of healthy vertices in $H$ is strictly greater than $h$ (i.e., $|H| > h$). Let $F = V(C[i]) \setminus H$. Since $|F| + |H| = |V(C[i])|$, it suffices to prove that $|F| + h < |V(C[i])|$. By the definition of $F$, for any vertex $y \in F$, at least one of $y$ and $y^j$ must be a faulty vertex. Thus, we can establish an upper bound of $|F|$ as follows:
\[
|F| \leq \ell + (\mu_i + \mu_j)(2n - 2).
\]
Since $\mu_i + \mu_j + 1 \leq \ell + 1 \leq n$, it implies that 

\begin{equation*}
\begin{split}
|F| + h &\leq \ell + (\mu_i + \mu_j)(2n - 2) + (2n - 2 - \ell - \mu_i - \mu_j)\\
&= (2n - 2)(\mu_i + \mu_j + 1) - (\mu_i + \mu_j + 1) + 1\\
&= (2n - 3)(\mu_i + \mu_j + 1) + 1\\
&\leq (2n - 3)n + 1.
\end{split}
\end{equation*}
Recall that the total number of vertices in $C[i]$ is $|V(C[i])| = \prod_{k=1,\, k \neq c}^n d_k \geq 3^{n-1}$. If $n \geq 4$, then $|F| + h \leq (2n - 3)n + 1 < 3^{n-1} \leq |V(C[i])|$, as desired. For $n = 3$, we have $\mu_i + \mu_j \leq \ell \leq 2$ and $|V(C[i])| \geq 3^{n-1} = 9$. If $|V(C[i])| \geq 11$, then $|F| + h \leq (2n - 3)n + 1 = 10 < 11 \leq |V(C[i])|$. We now consider $|V(C[i])|=9$ and omit the possibility of $|V(C[i])|=10$ because $d_k \geq 3$ is an integer for all $1 \leq k \leq n$ and $k \neq c$. In this case, $C[i]$ is isomorphic to $C(3,3)$. If $\mu_i + \mu_j \leq 1$, then $|F| + h \leq (2n - 3)(\mu_i + \mu_j + 1) + 1 \leq 2(2n - 3) + 1 = 7 < 9 = |V(C[i])|$.

If $\mu_i + \mu_j = 2$, we need to re-estimate $|F|$ to obtain a more accurate result. Without loss of generality, assume that $s_1$ and $s_2$ are the two faulty source vertices. Consider the following scenarios:

If $s_1\in V(C[i])$ and $s_2\in V(C[j])$ are adjacent, then $|F| = 5$ (see Fig.~\ref{fig:healthy_vertex_pairs-2}(a)). 

If $s_1\in V(C[i])$ and $s_2\in V(C[j])$ are not adjacent, then $|F| \leq \max\{7,8\} = 8$ (see Fig.~\ref{fig:healthy_vertex_pairs-2}(b) and \ref{fig:healthy_vertex_pairs-2}(c)).

If both $s_1$ and $s_2$ in $C[i]$ or in $C[j]$ are adjacent, then $|F| = 7$, (see Fig.~\ref{fig:healthy_vertex_pairs-2}(d)).

If both $s_1$ and $s_2$ in $C[i]$ or in $C[j]$ are not adjacent, then $|F| = 8$, (see Fig.~\ref{fig:healthy_vertex_pairs-2}(e)). 

In summary, since we have $|F| \leq 8$ when $n=3$ and $\mu_i + \mu_j=\ell=2$, it follows that $|F| + h \leq 8 + 2n - 2 - \ell - \mu_i - \mu_j = 8 + 6 - 2 - 2 - 2 = 8 < 9 = |V(C[i])|$.
\end{proof}

\medskip
\begin{lemma} \label{lma:IDP}
Let $C=C(d_1,d_2,\ldots,d_n)$, where $n \ge 3$ and $d_k \ge 3$ for $1 \le k \le n$, and let $U\subset V(C)$ be a set of faulty source vertices with $|U|=\ell<n$. Let $C[i]$ and $C[j]$ be two adjacent subnetworks in $C$ with $\kappa(C[i] \ominus U_i) \geq 2n - 2 - 2\mu_i$ and $\kappa(C[j] \ominus U_j) \geq 2n - 2 - 2\mu_j$, where $U_i = U \cap V(C[i])$, $\mu_i = |U_i|$, $U_j = U \cap V(C[j])$, and $\mu_j = |U_j|$. If $x \in V(C[i])$ and $y \in V(C[j])$ are two healthy vertices, then there exist at least $m$ internally disjoint $(x, y)$-paths passing through only healthy vertices in $C[i]$ and $C[j]$, where
\[
m = \begin{cases} 
2n - 2\ell     & \text{if}\ \mu_i + \mu_j < \ell; \\
2n - 2\ell - 1 & \text{if}\ \mu_i + \mu_j = \ell. 
\end{cases}
\]
\end{lemma}

\begin{proof}
Suppose that $C(d_1,d_2,\ldots,d_n)$ is partitioned into $d_c$ subnetworks $C[0], C[1], \ldots, C[d_c-1]$ along a dimension $c$, where $1 \le c\le n$. Without loss of generality, let $x \in V(C[0])$ and $y \in V(C[1])$, and assume that $\mu_0 \geq \mu_1$. By Lemma \ref{lm:healthy-pairs}, there exist $h = 2n-2-\ell-\mu_0-\mu_1$ neighbors $x_1, x_2, \ldots, x_h$ of $x$ in $C[0]$ such that
\[
x_k, x_k^1 \notin N[U]\ \text{for}\ k \in \{1,2,\ldots,h\}.
\]
Let $X = \{x_1, x_2, \ldots, x_h\}$, $Y = \{x_1^1, x_2^1, \ldots, x_h^1\}$, and $F = \{v \in V(C)\colon\, v\ \text{is a faulty vertex}\}$. For each $k\in\{0,1,\ldots, d_c-1\}$, let $F_k = F \cap V(C[k])$, $U_k = U \cap V(C[k])$, and $\mu_k = \vert U_k \vert$. Consider the following three cases.

\textbf{Case 1:} $\mu_0 = \ell = n - 1$.

In this case, $\mu_k = 0$ for $k = 1,2, \ldots , d_c - 1$. Since $\mu_0 + \mu_1 = \ell$, we need to find $m = 2n - 2\ell - 1 = 1$ path connecting $x$ and $y$. Since all faulty source vertices are located in $C[0]$ and $x$ is a healthy vertex, it follows that $x^1$ is also healthy in $C[1]$ and $|F_1|=\mu_0=n-1$. By the assumption $\kappa(C[1])= \kappa(C[1] \ominus U_1) \ge 2n - 2 - 2\mu_1 = 2n - 2 > n >|F_1|$ for $n \ge 3$ and Lemma~\ref{lm:disjoint-paths}, there exists an $(x^1, y)$-path $P_{x^1y}$ in $C[1]-F_1$. Let $P = \langle x, P_{x^1y} \rangle$, which is the required path passing through only healthy vertices in $C[0]$ and $C[1]$ that connects $x$ and $y$.

\textbf{Case 2:} $1 \le \mu_0 \le n-2$.

By Lemma~\ref{lm:healthy-pairs}, $C[0]$ contains a vertex $z$ that is not in $X$ and is healthy along with its outer neighbor $z^1$. We claim that there exists an $(x, z)$-path $P_{xz}$ in $C[0]-F_0$ that does not pass through any vertex of $X$. If $x^1$ is a healthy vertex, we can set $z = x$, making the path $P_{xz} = P_{xx} = \langle x \rangle$. In contrast, if $x^1$ is not a healthy vertex, we observe that $x^1 \notin U$ and $N(x^1) \cap U \neq \varnothing$. We will consider two scenarios based on this observation.

\textbf{Case 2.1:} $d_c = 3$.

In this case, the fault of $x^1$ is not caused by its outer neighbor $x^2$ since $x$ is healthy. This indicates that the fault of $x^1$ is due to a faulty source vertex in $C[1]$. Thus, $\mu_1 \geq 1$. Note that $\mu_0+\mu_1+\mu_2=\ell$, and we also have $|F_0|-|N_{C[0]}[U_0]|\le \mu_1+\mu_2$, where the term on the left represents the number of faulty vertices in $C[0] \ominus U_0$. Given the assumption of the connectivity of $C[0] \ominus U_0$, we proceed with the analysis
\begin{equation*}
\begin{split}
\kappa(C[0] \ominus U_0) 
&\ge 2n - 2 - 2\mu_0\\
&\ge 2n - 2 - 2\mu_0 - \mu_1 + 1\\
&= (2n - 2 - \ell - \mu_0 - \mu_1) + \mu_1 + \mu_2 +1\\
&= h + \mu_1 + \mu_2 + 1 \\
&> |X| + (\mu_1 + \mu_2).
\end{split}
\end{equation*}
Hence, by Lemma \ref{lm:disjoint-paths}, there exists an $(x, z)$-path $P_{xz}$ in $C[0]-(F_0 \cup X)$.

\textbf{Case 2.2:} $ d_c \geq 4 $.

The fault of $x^1$ is caused by faulty source vertices in $C[1]$ or $C[2]$. Thus, $\mu_1 + \mu_2 \geq 1$. Note that $\ell=\sum_{k=0}^{d_c-1} \mu_k$ and $|F_0|-|N_{C[0]}[U_0]|\le \mu_1+\mu_{d_c-1}$. Given that $\sum_{k=1}^{d_c-2} u_k \geq u_1 + u_2 \geq 1$, and based on the assumption of the connectivity of $C[0] \ominus U_0$, we have
\begin{equation*}
\begin{split}
\kappa(C[0] \ominus U_0) 
&\ge 2n - 2 - 2\mu_0\\
&\ge 2n - 2 - 2\mu_0 + 1 - \sum_{k=1}^{d_c-2} \mu_k\\
&= (2n - 2 - \ell - \mu_0 - \mu_1) + \mu_1 + \mu_{d_c-1} +1\\
&= h + \mu_1 + \mu_{d_c-1} + 1 \\
&> |X| + (\mu_1 + \mu_{d_c-1}).
\end{split}
\end{equation*}
Hence, by Lemma \ref{lm:disjoint-paths}, there exists an $(x, z)$-path $P_{xz}$ in $C[0]-(F_0 \cup X)$.

From the above, we demonstrate that there exists an $(x, z)$-path $P_{xz}$ in $C[0]-(F_0 \cup X)$. Let $Y' = Y \cup \{z^1\}$, where all vertices in $Y'$ are healthy. Note that $\ell=\sum_{k=0}^{d_c-1} \mu_k$ and $|F_1|-|N_{C[1]}[U_1]|\le \mu_0+\mu_2$. Since $\mu_0 \geq 1$ and $\sum_{k=3}^{d_c-1} \mu_k \geq 0$, by the assumption of the connectivity of $C[1] \ominus U_1$, we can further establish that
\begin{equation*}
\begin{split}
\kappa(C[1] \ominus U_1) 
&\ge 2n - 2 - 2\mu_1\\
&\ge 2n - 2 - 2\mu_1 - \mu_0 + 1 - \sum_{k=3}^{d_c-1} \mu_k\\
&= (2n - 2 - \ell - \mu_0 - \mu_1) + 1 + \mu_0 + \mu_2\\
&= h + 1 + \mu_0 + \mu_2\\
&\ge |Y'| + (\mu_0 + \mu_2).
\end{split}
\end{equation*}
By Lemma \ref{lm:disjoint-paths}, there exists a $(Y', y)$-fan in $C[1]-F_1$, denoted these $h+1$ internally disjoint paths as $P_{x_1^1y}, P_{x_2^1y}, \ldots, P_{x_h^1y}, P_{z^1y}$. Let $P_k = \langle x, x_k, P_{x_k^1y} \rangle$ for $k = 1,2, \ldots, h$, and let $P_{h+1} = \langle P_{xz}, P_{z^1y} \rangle$. Thus, $\{P_k\}_{k=1}^{h+1}$ forms $h + 1$ internally disjoint $(x, y)$-paths passing through only healthy vertices in $C[0]$ and $C[1]$ that connects $x$ and $y$. Moreover, we have
\begin{equation*}
\begin{split}
h + 1 &= 2n - 2 - \ell - \mu_0 - \mu_1 + 1 \\
&=
\begin{cases}
2n - \ell - (\mu_0 + \mu_1 + 1) \geq 2n - 2\ell & \text{if}\ \mu_0 + \mu_1 + 1 \leq \ell; \\
2n - 2\ell - 1 & \text{if}\ \mu_0 + \mu_1 = \ell.
\end{cases}
\end{split}
\end{equation*}
Therefore, $h + 1 \geq m$.

\medskip
\textbf{Case 3:} $\mu_0 = 0$.

In this case, $\mu_0 = \mu_1 = 0 < \ell$, which means we need to find $m = 2n - 2\ell$ internally disjoint paths connecting $x$ and $y$. Since the faults in $C[1]$ are caused by the faulty source vertices in $C[2]$, it follows that $|F_1|\le \ell$. By the assumption of the connectivity of $C[1] \ominus U_1$, we have
\begin{equation*}
\begin{split}
\kappa(C[1]) &= \kappa(C[1] \ominus U_1) \\
&\ge 2n - 2 \\
&\ge (2n - 2 - \ell - \mu_0 - \mu_1) + \ell\\
&= |Y| + \ell.
\end{split}
\end{equation*}
By Lemma \ref{lm:disjoint-paths}, there exists a $(Y, y)$-fan in $C[1]-F_1$, denoted these $h$ internally disjoint paths as $P_{x_1^1y}, P_{x_2^1y}, \ldots, P_{x_h^1y}$. Let $P_k = \langle x, x_k, P_{x_k^1y} \rangle$ for $k = 1,2, \ldots, h$. Thus, $\{P_k\}_{k=1}^{h}$ forms $h= 2n - 2 - \ell$ internally disjoint $(x, y)$-paths passing through only healthy vertices in $C[0]$ and $C[1]$ that connects $x$ and $y$.

If $\ell \geq 2$, then $h = 2n - 2 - \ell \geq 2n - 2\ell = m$. Hence, we have successfully proven the lemma.

If $\ell = 1$ and $d_c \geq 4$, then $h=2n-3$ and at least one of $U_2$ and $U_{d_c - 1}$ must be empty. Without loss of generality, assume $U_2 = \varnothing$, i.e., $\mu_2 = 0$ and $\mu_{d_c - 1} \leq 1$. Thus, all vertices in $C[1]$ are healthy, i.e., $F_1=\varnothing$. Let $Y' = Y \cup \{x^1\}$. By Lemma~\ref{lm:connectivity}, $\kappa(C[1])=2(n-1)=h+1=|Y'|$. According to Lemma~\ref{lm:disjoint-paths}, there exists a $(Y', y)$-fan in $C[1]$, denoted these $h+1$ internally disjoint paths as $P_{x_1^1y}, P_{x_2^1y}, \ldots, P_{x_h^1y}, P_{x^1y}$. Let $P_k = \langle x, x_k, P_{x_k^1y} \rangle$ for $k = 1, 2, \ldots, h$, and let $P_{h+1} = \langle x, P_{x^1y} \rangle$. Fig.~\ref{fig:l1} illustrates these paths. Therefore, $\{P_k\}_{k=1}^{h+1}$ forms $h + 1$ internally disjoint $(x, y)$-paths passing through only healthy vertices in $C[0]$ and $C[1]$ that connects $x$ and $y$, and $h + 1 = 2n - 2\ell = m$. 

\begin{figure}[ht]
	\centering
	\includegraphics[width=0.5\textwidth]{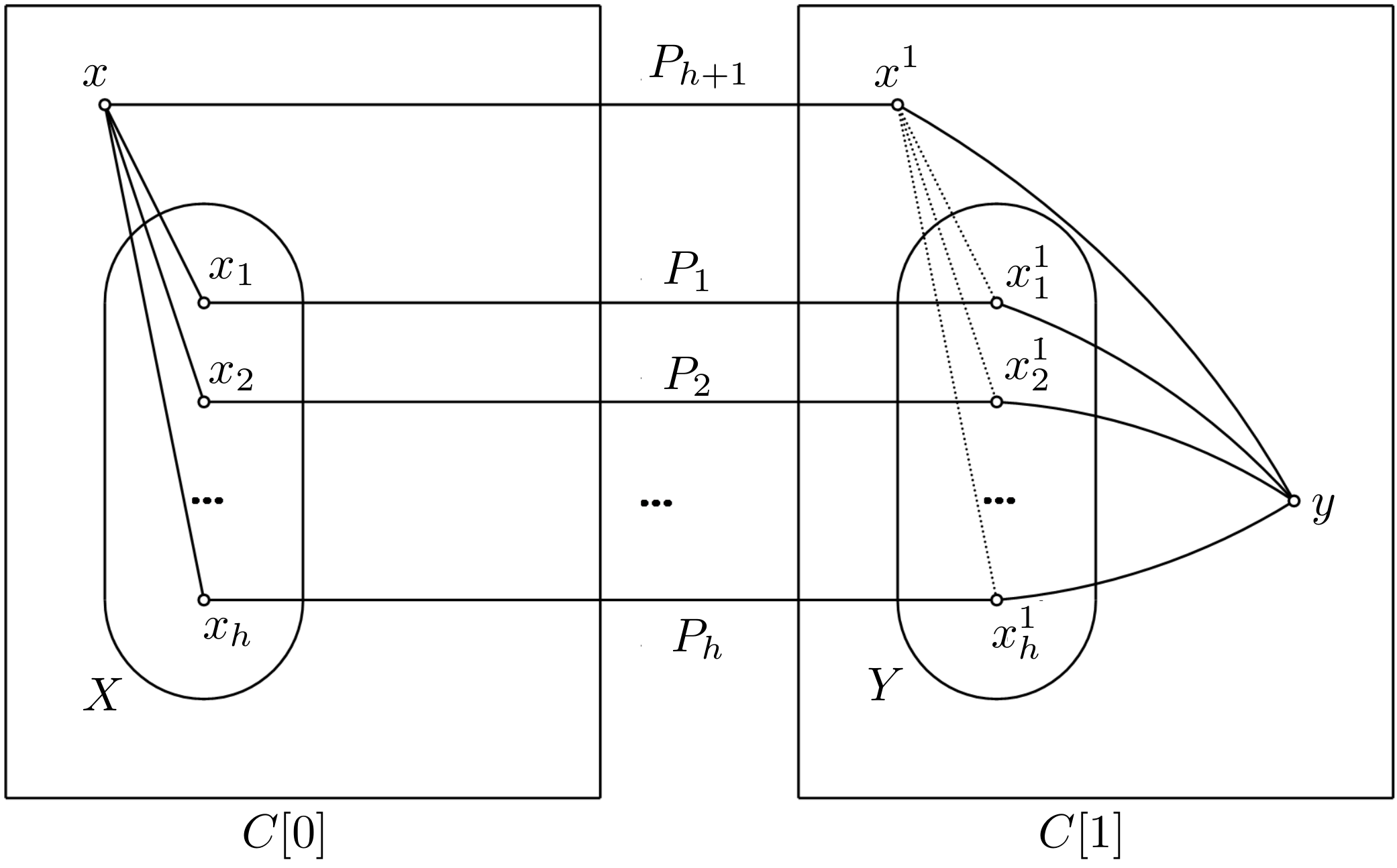}
	\caption{The illustration when $\ell=1$ and $d_c \ge 4$}
	\label{fig:l1}
\end{figure}

If $\ell = 1$ and $d_c = 3$, then $h=2n-3$ and $\mu_2 = |F_0|=|F_1|=1$. In this case, $x^1$ and $y^0$ are both healthy vertices. If $y = x^1$, then $\langle x, y = x^1 \rangle$ is an $(x, y)$-path. Let $P_k = \langle x, x_k, x_k^1, y = x^1 \rangle$ for $k = 1, 2, \ldots, h$, and let $P_{h+1} = \langle x, y = x^1 \rangle$. Therefore, $\{P_k\}_{k=1}^{h+1}$ forms $h + 1 = 2n-2\ell = m$ internally disjoint $(x, y)$-paths passing through only healthy vertices in $C[0]$ and $C[1]$ that connects $x$ and $y$ (see Fig.~\ref{fig:l2}).

\begin{figure}[htbp]
	\centering
	\includegraphics[width=0.5\textwidth]{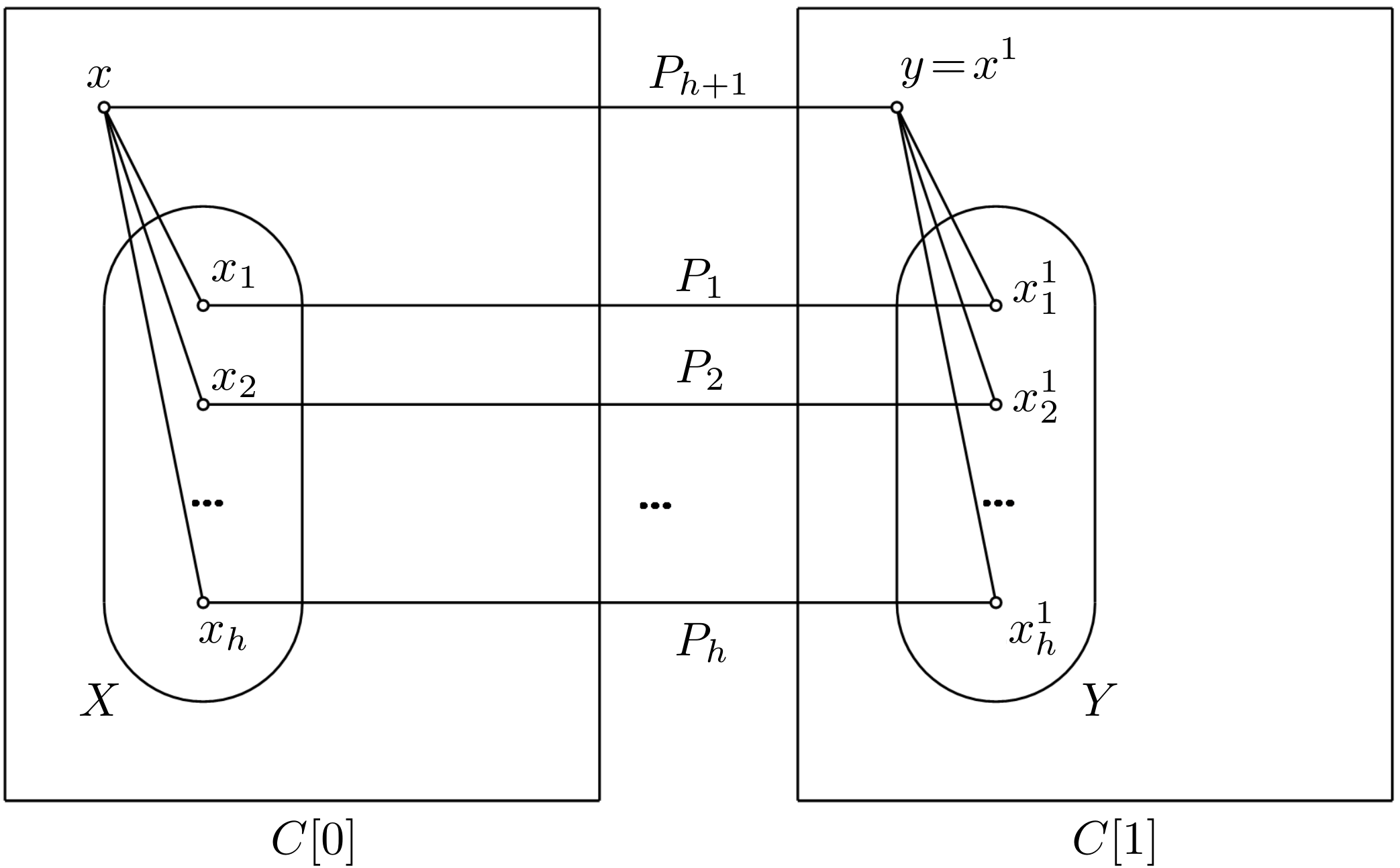}
	\caption{The illustration when $\ell=1$, $d_c=3$, and $y=x^1$}
	\label{fig:l2}
\end{figure}

If $y \neq x^1$, then $y^0 \neq x$. If $y^0 \in X$, without loss of generality, let $y^0 = x_h$ (i.e., $y=x_h^1$) and $Y' = (Y \setminus \{x_h^1\}) \cup \{x^1\}$. Recall that $|F_1|=1$. By Lemma~\ref{lm:connectivity}, $\kappa(C[1])=2(n-1)=h+1=|Y'|+|F_1|$. According to Lemma~\ref{lm:disjoint-paths}, there exists a $(Y', y)$-fan in $C[1]-F_1$, denoted these $h$ internally disjoint paths as $P_{x_1^1y}, P_{x_2^1y}, \ldots, P_{x_{h-1}^1y}, P_{x^1y}$. Let $P_k = \langle x, x_k, P_{x_k^1y} \rangle$ for $k = 1, 2, \ldots, h - 1$, $P_h = \langle x, x^1,y \rangle$, and $P_{h+1} = \langle x, x_h = y^0, y \rangle$. Therefore, $\{P_k\}_{k=1}^{h+1}$ forms $h + 1 = 2n-2\ell = m$ internally disjoint $(x, y)$-paths passing through only healthy vertices in $C[0]$ and $C[1]$ that connects $x$ and $y$ (see Fig.~\ref{fig:l3}).

\begin{figure}[htbp]
	\centering
	\includegraphics[width=0.5\textwidth]{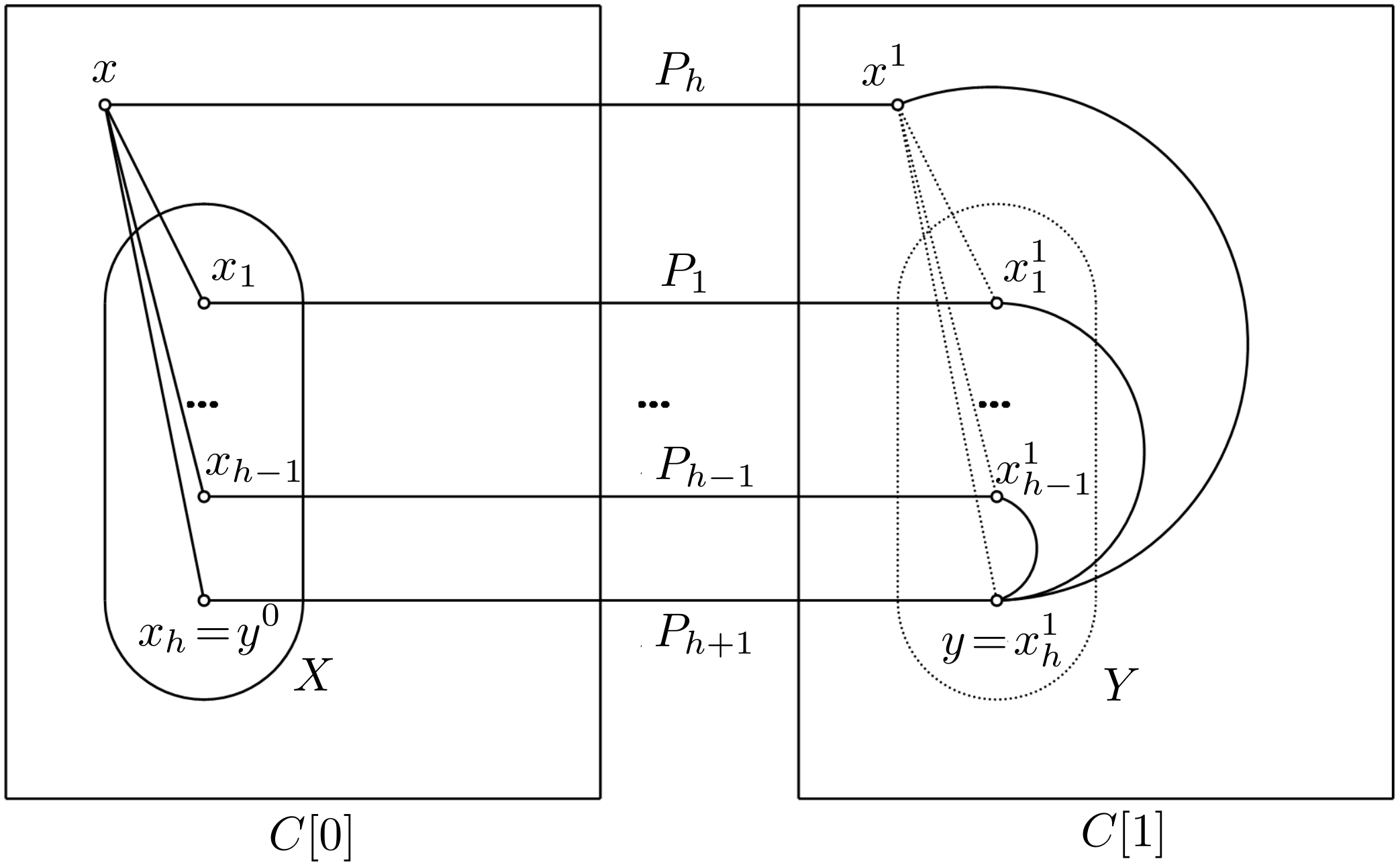}
	\caption{The illustration when $\ell=1$, $d_c=3$ and $y \in Y$}
	\label{fig:l3}
\end{figure}

\begin{figure}[ht]
	\centering
	\includegraphics[width=0.5\textwidth]{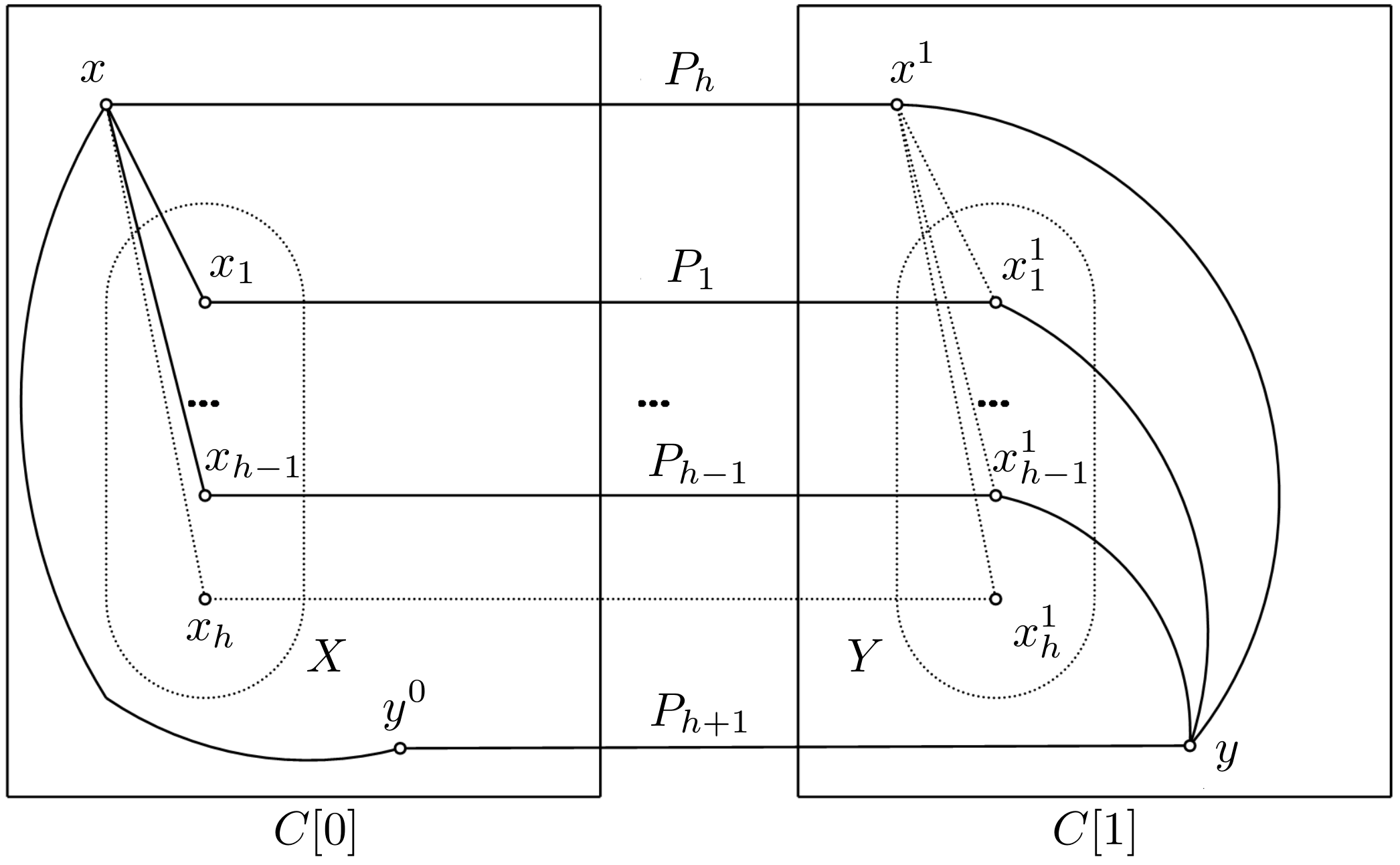}
	\caption{The illustration when $\ell=1$, $d_c=3$ and $y \in V(C[1]) \setminus (\{x^1\} \cup Y)$}
	\label{fig:l4}
\end{figure}

It remains to consider $y^0 \notin X\cup\{x\}$ (i.e., $y\notin Y\cup\{x^1\}$). Recall that $|F_0|=|F_1|=1$. Let $X' = X \setminus \{x_h\}$. By Lemma~\ref{lm:connectivity}, $\kappa(C[0])=2(n-1)=h+1=(|X'|+1)+|F_0|$. According to Lemma~\ref{lm:disjoint-paths}, there exists an $(x, y^0)$-path $P_{xy^0}$ in $C[0]-(X'\cup F_0)$. Let $Y' = (Y \setminus \{x_h^1\}) \cup \{x^1\}$. By Lemma~\ref{lm:connectivity}, $\kappa(C[1]) = 2(n-1)=h+1=|Y'| + |F_1|$. By Lemma \ref{lm:disjoint-paths}, there exists a $(Y', y)$-fan in $C[1]-F_1$, denoted these $h$ internally disjoint paths as $P_{x_1^1y}, P_{x_2^1y}, \ldots, P_{x_{h-1}^1y}, P_{x^1y}$. Let $P_k = \langle x, x_k, P_{x_k^1y} \rangle$ for $k = 1, 2, \ldots, h - 1$, and $P_h = \langle x, P_{x^1y} \rangle$, and $P_{h+1} = \langle P_{xy^0}, y \rangle$. Therefore, $\{P_k\}_{k=1}^{h+1}$ forms $h + 1 = 2n-2\ell = m$ internally disjoint $(x, y)$-paths passing through only healthy vertices in $C[0]$ and $C[1]$ that connects $x$ and $y$ (see Fig.~\ref{fig:l4}). 
\end{proof}

\medskip
\begin{theorem} \label{thm:lower-connectivity}
Let $C=C(d_1,d_2,\ldots,d_n)$, where $n \ge 2$ and $d_i \ge 3$ for $1 \le i \le n$. If $U\subset V(C)$ is a set of faulty source vertices with $0\leq |U|=\ell\leq n$, then $\kappa(C \ominus U) \geq 2n - 2\ell$.
\end{theorem}

\begin{proof}
According to Lemma 2.3, the result holds when $\ell=0$. Since connectivity cannot be negative, it also applies when $\ell=n$. Hence, we only need to consider $0 < \ell < n$. We will proceed with this proof using induction on $n$.

When $n = 2$, the only possible choice is $\ell=1$ and $C = C(d_1, d_2)$ is a $d_1 \times d_2$ mesh. Since $C(d_1,d_2)$ is vertex-transitive, without loss of generality, assume that $00$ is the unique faulty source vertex. Fig.~\ref{fig:del_00} illustrates the network obtained by removing vertex $00$ and its neighbors from $C$. As the resulting network is neither empty nor complete, and it is connected, we have $\kappa(C \ominus U) \geq 2 = 2n - 2\ell$.

\begin{figure}[htbp]
	\centering
	\includegraphics[width=0.75\textwidth]{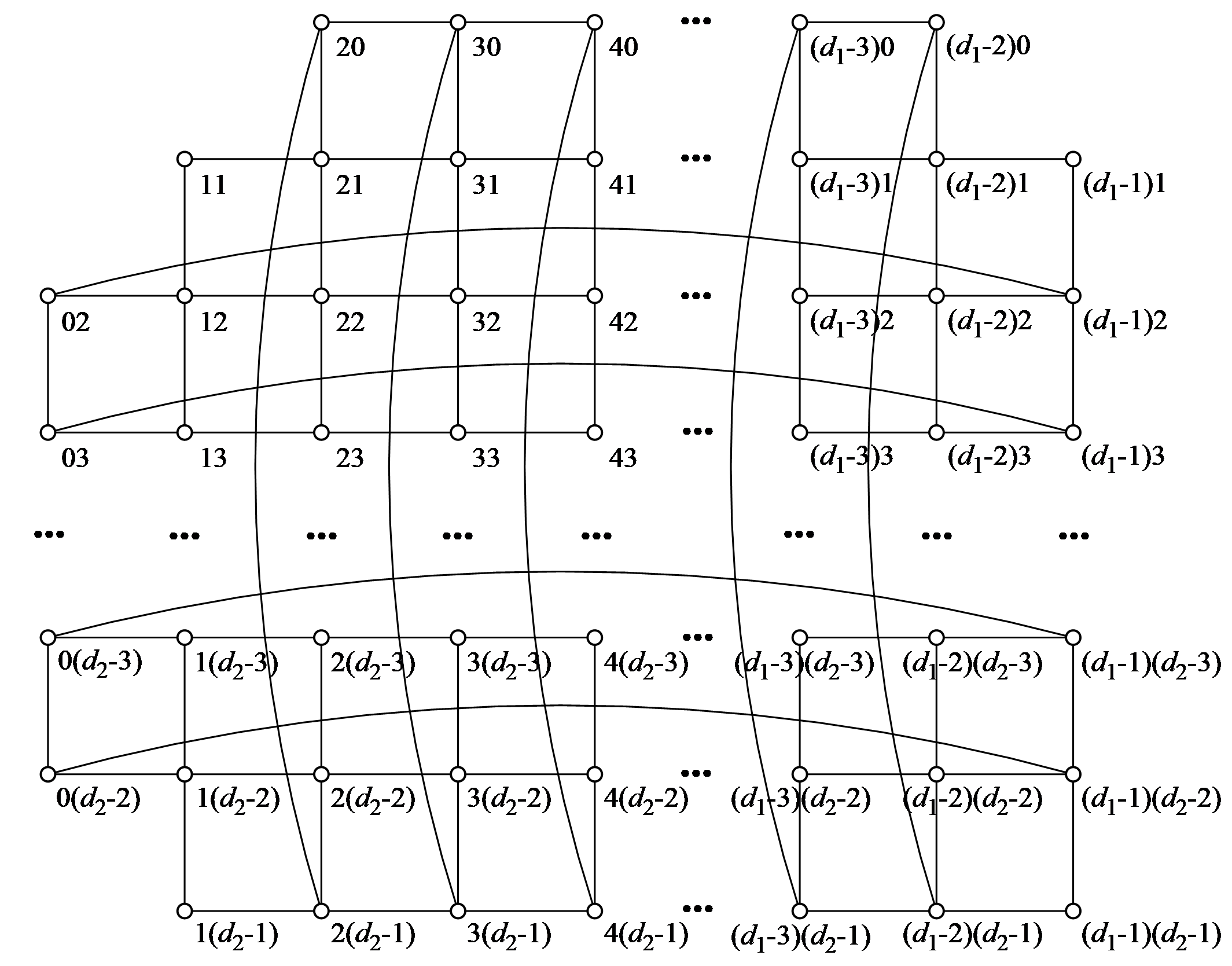}
	\caption{A 2-dimensional undirected toroidal mesh after removing vertex $00$ and its neighbors. Note that the first row does not exist when $d_1 = 3$; and the first column does not exist when $d_2 = 3$.}
	\label{fig:del_00}
\end{figure}

Let $k \geq 2$ be an integer and suppose that the theorem holds when $n = k$, i.e., $\kappa(C \ominus U) \geq 2k - 2\ell$ for any set $U \subseteq V(C)$ of size $\ell$, where $0 \le \ell \le k$. We now consider  $n = k + 1$ and there is a set $U$ of faulty source vertices with $0<|U|=\ell<k+1$. To prove $\kappa(C \ominus U) \geq 2(k + 1) - 2\ell$, it suffices to show that for any two distinct vertices $x,y \in V(C \ominus U)$, there exist at least $2k + 2 - 2\ell$ internally disjoint $(x, y)$-paths in $C \ominus U$. Note that we can always find an integer $c$ with $1 \leq c \leq n=k + 1$ such that $C$ is partitioned into $d_c$ subnetworks, denoted as $C[0],C[1],\ldots,C[d_c-1]$, along the $c$-th dimension, and $x$ and $y$ reside in different subnetworks. Denote
\[
F = \{v \colon\, v \in N[U] \}\ \text{and}\ F_i = F \cap V(C[i])\ \text{for}\ 0 \le i \le d_c-1,
\]
and 
\[
U_i = U \cap V(C[i])\ \text{and}\ \mu_i = |U_i|\ \text{for}\ 0 \le i \le d_c-1.
\]
By induction hypothesis, for $i = 0, 1, \ldots, d_c - 1$, we have $\kappa(C[i] \ominus U_i) \geq 2k - 2\mu_i$, where $0\le \mu_i \leq \ell \leq k$. Consider the following two cases.

\medskip
\textbf{Case 1:} $x$ and $y$ are located in adjacent subnetworks.

Without loss of generality, let $x \in V(C[0])$ and $y \in V(C[1])$, and assume that $\mu_0 \ge \mu_1$. By Lemma~\ref{lma:IDP}, there exist at least $m$ internally disjoint $(x, y)$-paths passing through only healthy vertices in $C[0]$ and $C[1]$, denoted as $P_1, P_2, \ldots, P_m$, where $m$ is defined as follows:

If $\mu_0 + \mu_1 < \ell$, then $m = 2n - 2\ell = 2k + 2 - 2\ell$. Thus, these paths are internally disjoint that connect $x$ and $y$ in $C \ominus U$, satisfying the desired requirement.

If $\mu_0 + \mu_1 = \ell$, then $m = 2n - 2\ell - 1 = 2k + 1 - 2\ell$. In this case, we need to construct another $(x, y)$-path in $C \ominus U$ that is internally disjoint from $P_1, P_2, \ldots, P_m$. Note that $\mu_i = 0$ for $i = 2, 3, \ldots, d_c - 1$. Since both $x$ and $y$ are healthy vertices, $x^2, x^3, \ldots, x^{d_c-1}, y^2, y^3, \ldots, y^{d_c-1}$ must also be healthy vertices. Since $|F_2|\le \ell\le k$ and $k\ge 2$, we have $\kappa(C[2]) = \kappa(C[2] \ominus U_2) \ge 2k - 2\mu_2 = 2k\ge |F_2|+1$. By Lemma \ref{lm:disjoint-paths}, there exists an $(x^2, y^2)$-path $P_{x^2y^2}$ in $C[2]-F_2$. Let
\[
P_{m+1} = 
\begin{cases} 
\langle x, P_{x^2y^2}, y \rangle & \text{if}\ d_c = 3; \\
\langle x, x^{d_c-1}, x^{d_c-2}, \ldots, x^3, P_{x^2y^2}, y \rangle & \text{if}\ d_c \ge 4.
\end{cases}
\]
Thus, $\{P_i\}_{i=1}^{m+1}$ forms $m + 1 = 2k + 2 - 2\ell$ internally disjoint $(x, y)$-paths in $C \ominus U$.

\medskip
\textbf{Case 2:} $x$ and $y$ are located in non-adjacent subnetworks.

In this case, $d_c \geq 4$. Without loss of generality, let $x \in V(C[0])$ and $y \in V(C[t])$ for $2 \leq t \leq d_c - 2$, and assume that $\mu_0 \geq \mu_t$. According to Lemma~\ref{lm:healthy-vertex}, $C[i]$ contains at least one healthy vertex, say $v_i$, for each $i = 0,1, 2,\ldots, d_c-1$. Let $v_0 = x$ and $v_t = y$. By Lemma~\ref{lma:IDP}, for each $i = 0, 1, \ldots, t - 1$, we can find at least $m$ internally disjoint $(v_i, v_{i+1})$-paths passing through only healthy vertices in $C[i]$ and $C[i + 1]$ that connect $v_i \in V(C[i])$ and $v_{i+1} \in V(C[i + 1])$, denoted these $m$ internally disjoint paths as $\{P_{v_iv_{i+1}}^j\}_{j=1}^m$, where
\[
m = 
\begin{cases} 
2n - 2\ell = 2k + 2 - 2\ell & \text{if}\ \mu_i + \mu_{i+1} < \ell; \\
2n - 2\ell - 1 = 2k + 1 - 2\ell & \text{if}\ \mu_i + \mu_{i+1} = \ell. 
\end{cases} 
\]

Note that for $i = 1, 2, \ldots, t - 1$, two sets of paths $\{P_{v_{i-1}v_i}^j\}_{j=1}^m$ and $\{P_{v_iv_{i+1}}^j\}_{j=1}^m$ may share common vertices in $C[i]$, except for the vertex $v_i$. Here, we describe some notations that appear in these paths. For more details, readers can refer to Fig.~\ref{fig:paths-1}, which is explained as follows:

\begin{enumerate}[(1)]
\item \label{1}
For $j=1,2,\ldots,m$, let $P_{(v_0=x)v_1}^j = \langle P_{x v_{0j}}, P_{u_{1j}v_1} \rangle$. Here, $P_{x v_{0j}}$ is the path connecting $x$ to a vertex $v_{0j}$ in $C[0]-F_0$, while $P_{u_{1j}v_1}$ connects a vertex $u_{1j}$ to $v_1$ in $C[1]-F_1$. The healthy vertices $v_{0j}$ and $u_{1j}$ are outer neighbors to each other across $C[0]$ and $C[1]$. Refer to Description~(1) in Fig.~\ref{fig:paths-1}.
\item \label{2}
For $j=1,2,\ldots,m$, let $P_{v_{t-1}(v_t=y)}^j = \langle P_{v_{t-1}v_{(t-1)j}}, P_{u_{tj}y} \rangle$. Here, $P_{v_{t-1}v_{(t-1)j}}$ is the path connecting $v_{t-1}$ to a vertex $v_{(t-1)j}$ in $C[t-1]-F_{t-1}$, while $P_{u_{tj}y}$ connects a vertex $u_{tj}$ to $y$ in $C[t]-F_t$. The healthy vertices $v_{(t-1)j}$ and $u_{tj}$ are outer neighbors to each other across $C[t-1]$ and $C[t]$. Refer to Description~(2) in Fig.~\ref{fig:paths-1}.
\item \label{3}
For $i = 1,2,\ldots,t-2$ and $j=1,2,\ldots,m$, let $P_{v_iv_{i+1}}^j = \langle P_{v_i v_{ij}}, P_{u_{(i+1)j}v_{i+1}}\rangle$. Here, $P_{v_i v_{ij}}$ is the path connecting $v_i$ to a vertex $v_{ij}$ in $C[i]-F_i$, while $P_{u_{(i+1)j}v_{i+1}}$ connects a vertex $u_{(i+1)j}$ to $v_{i+1}$ in $C[i+1]-F_{i+1}$. The healthy vertices $v_{ij}$ and $u_{(i+1)j}$ are outer neighbors to each other across $C[i]$ and $C[i+1]$. Refer to Description~(3) in Fig.~\ref{fig:paths-1}.
\end{enumerate}

\begin{figure}[htbp]
	\centering
	\includegraphics[width=1.00\textwidth]{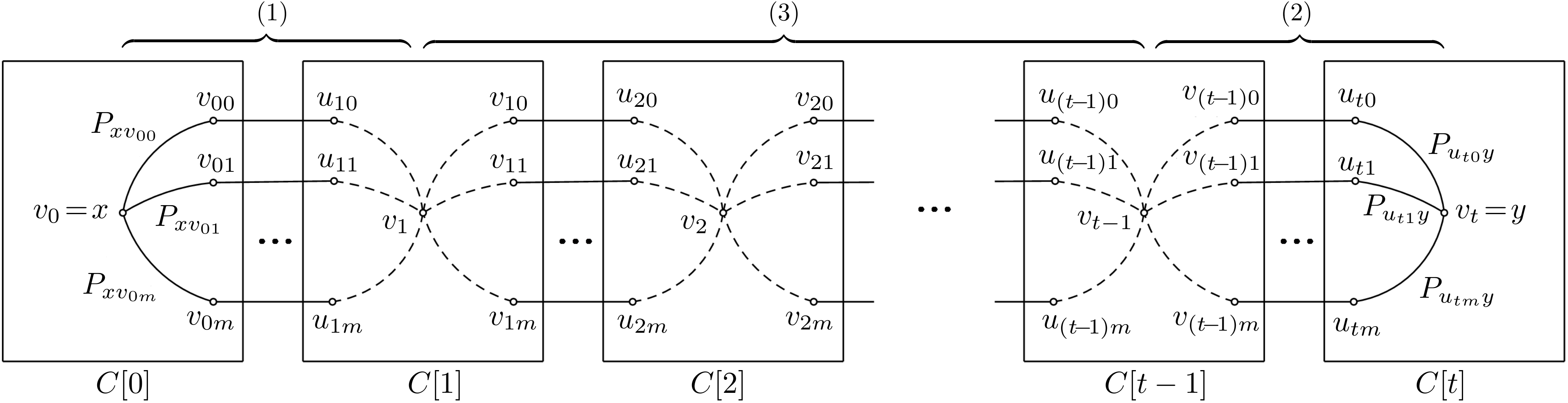}
	\caption{The illustration of three descriptions of paths in Case 2 of Theorem~\ref{thm:lower-connectivity}.}
	\label{fig:paths-1}
\end{figure}

In what follows, we establish the desired paths by consider the following two cases.

\textbf{Case 2.1:} $\mu_0 < \ell$.

Recall that $\mu_t \leq \mu_0$. When $d_c \geq 5$ or when $d_c = 4$ and $\mu_0 + \mu_t < \ell$, we can always assume that the condition $\mu_{i-1} + \mu_i + \mu_{i+1} < \ell$ holds for $i = 1,2, \ldots, t - 1$. However, if this assumption does not hold, we can replace $C[0], C[1], \ldots, C[t]$ with $C[0], C[d_c - 1], C[d_c - 2], \ldots, C[t]$, while still denoting them as $C[0], C[1], \ldots, C[t]$ (i.e., we choose another reverse sequence of subnetworks in this cycle from $C[0]$ to $C[t]$). This adjustment ensures that the assumption remains valid. For each $i$ with $0 < i < t$, let $X_i = \{u_{ij}\}_{j=1}^m$ and $Y_i = \{v_{ij}\}_{j=1}^m$, where $m = 2k + 2 - 2\ell$ (because $\mu_{i-1} + \mu_i < \ell$ and $\mu_i + \mu_{i+1} < \ell$). Since $C[i] \ominus U_i$ contains at most $\mu_{i-1} + \mu_{i+1}$ faulty vertices, by induction hypothesis, we can obtain that
\begin{equation*}
\begin{split}
\kappa(C[i] \ominus U_i) &\ge 2k-2\mu_i \\
&= m-2+2\ell-2\mu_i \\ 
&\ge m-2+2(\mu_{i-1} + \mu_i + \mu_{i+1} + 1)-2\mu_i \\
&= m+2(\mu_{i-1} + \mu_{i+1}) \\
&\ge |X_i|+(\mu_{i-1} + \mu_{i+1}).
\end{split}
\end{equation*}
By Lemma \ref{lm:disjoint-paths}, there exist $m = 2k + 2 - 2\ell$ pairwise disjoint $(X_i, Y_i)$-paths in $C[i]-F_i$, denoted these paths as $\{P_{u_{ij}v_{ij}}\}_{j=1}^m$. For $j = 1, 2, \ldots, m$, let 
\[
P_j = \langle P_{xv_{0j}}, P_{u_{1j}v_{1j}}, P_{u_{2j}v_{2j}}, \ldots, P_{u_{(t-1)j}v_{(t-1)j}}, P_{u_{tj}y} \rangle.
\]
Fig.~\ref{fig:paths-2} illustrates these paths. Thus, $\{P_j\}_{j=1}^m$ forms $m = 2k + 2 - 2\ell$ internally disjoint $(x, y)$-paths in $C \ominus U$.

\begin{figure}[htbp]
	\centering
	\includegraphics[width=1.00\textwidth]{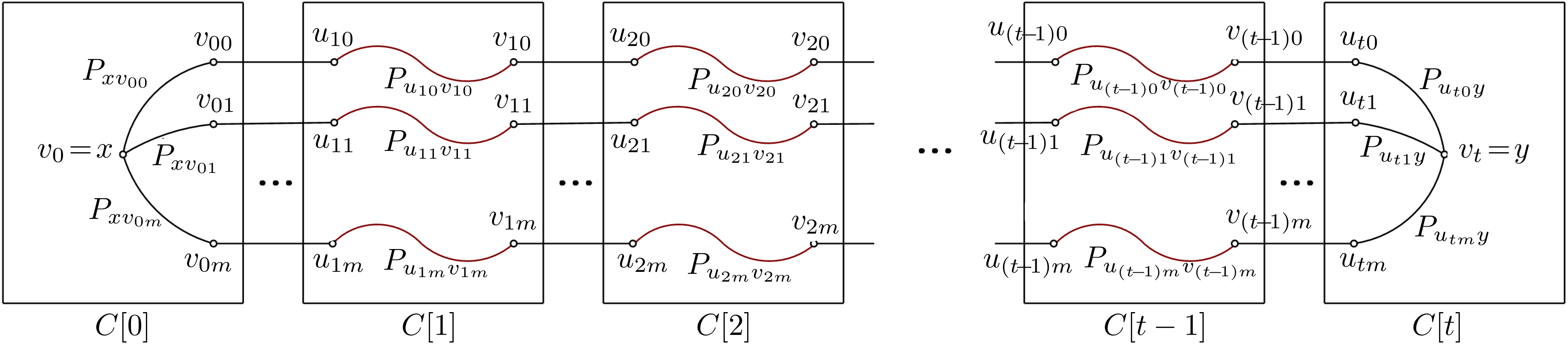}
	\caption{The illustration of Case 2.1, in the condition where $d_c \geq 5$, or $d_c = 4$ and $\mu_0 + \mu_t < \ell$}
	\label{fig:paths-2}
\end{figure}

When $d_c = 4$ and $\mu_0 + \mu_t = \ell$, the only possible choice is $t=2$. That is, $\mu_0 + \mu_2 = \ell$ and $\mu_1 = \mu_3 = 0$. Recall that we have $\mu_2 \leq \mu_0 < \ell$, which implies that $\ell \ne 1$ (i.e., $\ell \ge 2$). In this case, let $X = \{u_{1j}\}_{j=1}^m$ and $Y = \{v_{1j}\}_{j=1}^m$, where $m = 2k + 2 - 2\ell$.  Since the faults in $C[1]$ arise from the faulty source vertices located in $C[0]$ and $C[2]$, we have $|F_1| \le \mu_0 + \mu_2 = \ell$. Also, we have $\kappa(C[1]) = \kappa(C[1] \ominus U_1) = 2k - 2\mu_1 = 2k=m-2+2\ell\geq |X|+\ell$. By Lemma \ref{lm:disjoint-paths}, there exist $m = 2k + 2 - 2\ell$ pairwise disjoint $(X, Y)$-paths in $C[1]-F_1$, denoted by $\{P_{u_{1j}v_{1j}}\}_{j=1}^m$. For $j = 1,2, \ldots, m$, let $P_j = \langle P_{xv_{0j}}, P_{u_{1j}v_{1j}}, P_{u_{2j}y} \rangle$. Thus, $\{P_j \}_{j=1}^m$ forms $m = 2k + 2 - 2\ell$ internally disjoint $(x, y)$-paths in $C \ominus U$.
 
\textbf{Case 2.2:} $\mu_0 = \ell$.

In this case, we have $\mu_i = 0$ for $i = 1,2, \ldots, d_c - 1$. For Description~(\ref{1}), since $\mu_0+\mu_1=\ell$, we have $m = 2k + 1 - 2\ell$, and one has that $m = 2k + 2 - 2\ell$ for Descriptions~(\ref{2}) and (\ref{3}). Here, we consider $m = 2k + 1 - 2\ell$ for consistency because we can choose the first $2k + 1 - 2\ell$ required paths for the latter. Let $X_0 = \{v_{0j}\}_{j=1}^m$ and $Y_t = \{u_{tj}\}_{j=1}^m$. Recall that there is an $(x, X_0)$-fan in $C[0]$ and there is a $(Y_t, y)$-fan in $C[t]$. For $0 < i < t$, let $X_i = \{u_{ij}\}_{j=1}^m$ and $Y_i = \{v_{ij}\}_{j=1}^m$. Since $C[i]\ominus U_i$ contains at most $\mu_{i-1} + \mu_{i+1}\leq \ell$ faulty vertices and $\ell \geq 1$, by induction hypothesis, we have 
\[
\kappa(C[i]) = \kappa(C[i] \ominus U_i) \geq 2k - 2\mu_i = 2k=m-1+2\ell\geq |X_i| + (\mu_{i-1} + \mu_{i+1}).
\]
By Lemma~\ref{lm:disjoint-paths}, there exist $m = 2k + 1 - 2\ell$ pairwise disjoint $(X_i, Y_i)$-paths in $C[i]-F_i$, denoted as $\{P_{u_{ij}v_{ij}}\}_{j=1}^m$. Since $\mu_{t-1}=\mu_t=\mu_{t+1}=0$, we have $F_t=\varnothing$. By Lemma~\ref{lm:connectivity}, since $\ell \geq 1$, we have
\[
\kappa(C[t])=2(n-1)=2k=m-1+2\ell\geq m+1.
\]
Moreover, by Lemma~\ref{lm:healthy-pairs}, at least $h+1$ healthy vertices are contained in $C[t]$ such that these $h+1$ vertices and their outer neighbors in $C[t+1]$ are also healthy, where
\begin{equation*}
\begin{split}
h + 1 &= (2n - 2 - \ell - \mu_t - \mu_{t+1}) + 1 \\
&= 2(k + 1) - 2 - \ell + 1 \\
&> 2k + 1 - 2\ell \\
&= m.
\end{split}
\end{equation*}
This implies that either $y$ or there is a vertex $z \notin \{u_{tj}\}_{j=1}^m\cup\{y\}$ in $C[t]$ such that its outer neighbor in $C[t+1]$ is healthy. Without loss of generality, assume both $z$ and $z^{t+1}$ are healthy. Since $\kappa(C[t])\geq m+1$, we can extend the existing $(Y_t, y)$-fan to a $(Y_t\cup\{z\}, y)$-fan in $C[t]$, which contains internally disjoint paths $P_{zy}$ and $P_{u_{tj}y}$ for $j = 1,2,\ldots,m$. Since $\mu_i = 0$ for $i = t+1,t+2, \ldots, d_c - 1$, all vertices $x^{d_c - 1}, x^{d_c - 2}, \ldots, x^{t+1}$ are healthy. Since $|F_{t+1}| \le \ell$, by induction hypothesis, it follows that 
\[
\kappa(C[t+1]) = \kappa(C[t+1] \ominus U_{t+1})\geq 2k - 2\mu_{t+1}=2k\geq 2\ell\geq |F_{t+1}|+1.
\]
By Lemma~\ref{lm:disjoint-paths}, there exists an $(x^{t+1}, z^{t+1})$-path $P_{x^{t+1}z^{t+1}}$ in $C[t+1]-F_{t+1}$. For $j = 1,2, \ldots, m$, let
\[
P_j = \langle P_{xv_{0j}},P_{u_{1j}v_{1j}}, P_{u_{2j}v_{2j}}, \ldots, P_{u_{(t-1)j}v_{(t-1)j}},P_{u_{tj}y} \rangle.
\]
Also, let
\[
P_{m+1} = \langle x, x^{d_c-1}, x^{d_c-2}, \ldots, x^{t+2}, P_{x^{t+1}z^{t+1}}, P_{zy} \rangle.
\]
Thus, $\{P_j\}_{j=1}^{m+1}$ forms $m + 1 = 2k + 2 - 2\ell$ internally disjoint $(x, y)$-paths in $C \ominus U$ (see
Fig.~\ref{fig:illustration of case 2.2}). 

Thus, we demonstrate that for any two distinct vertices $x, y \in V(C \ominus U)$, there exist at least $2k + 2 - 2\ell$ internally disjoint $(x, y)$-paths in $C \ominus U$. This completes the proof.
\end{proof}

\begin{figure}[htbp]
	\centering
	\includegraphics[width=1.00\textwidth]{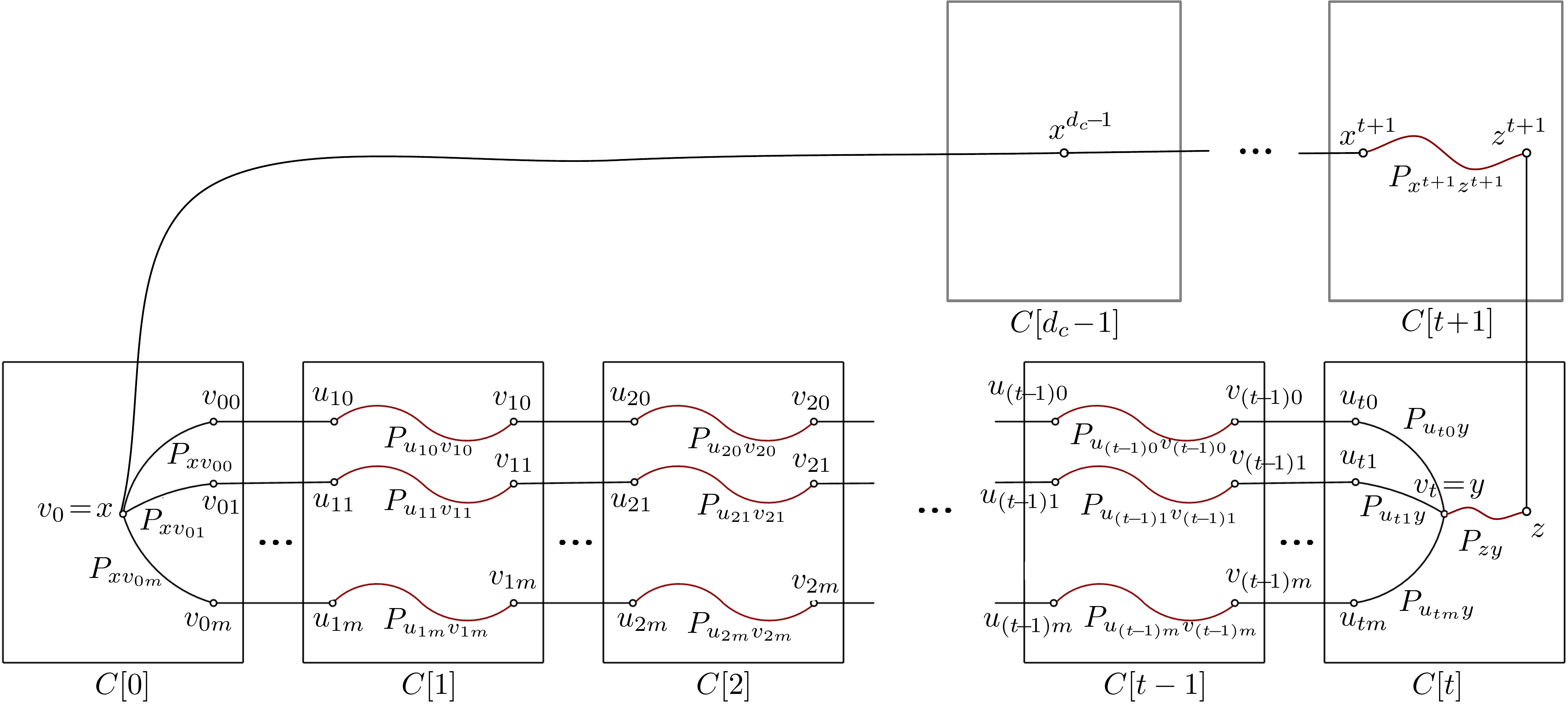}
	\caption{The illustration of Case 2.2}
	\label{fig:illustration of case 2.2}
\end{figure}

We are now in a position to deal with the proof of a lower bound of $\kappa_{NB}(C(d_1, d_2, \cdots, d_n))$.

\begin{theorem} \label{thm:lower-bound}
Let $C=C(d_1,d_2,\ldots,d_n)$, where $n \ge 2$ and $d_i \ge 3$ for $1 \le i \le n$. Then, $\kappa_{NB}(C) \geq n$.
\end{theorem}

\begin{proof}
Let $U$ be the set of any $\ell$ faulty source vertices in $C$, where $0 \leq \ell \leq n - 1$. Then, by Theorem \ref{thm:lower-connectivity}, $\kappa(C \ominus U) \geq 2n - 2\ell \geq 2$, indicating that $C \ominus U$ is connected. Since $|V(C \ominus U)| \geq (\prod_{i=1}^{n} d_i) - \ell (2n + 1) \geq 3^n - (n - 1)(2n + 1) \geq 4$, the survival graph $C \ominus U$ is not an empty graph. Furthermore, as the $n$-dimensional undirected toroidal mesh is a $K_4$-free graph, it follows that $C \ominus U$ is also not a complete graph. Therefore, $\kappa_{NB}(C) > n - 1$, i.e., $\kappa_{NB}(C) \geq n$.
\end{proof}

\subsection{Upper bound of $\kappa_{NB}(C(d_1,d_2,\ldots,d_n))$}
\label{Sec:Upperr-bound}

The following theorem provides an example to determine the upper bound of neighbor connectivity in an $n$-dimensional undirected toroidal mesh.

\begin{theorem} \label{thm:upper-bound}
Let $C=C(d_1,d_2,\ldots,d_n)$, where $n \ge 2$ and $d_i \ge 3$ for $1 \le i \le n$. Then, $\kappa_{NB}(C) \leq n$.
\end{theorem}

\begin{proof}
If $n = 2$ and $d_i = 3$ for $1 \leq i \leq 2$ (i.e., $C = C(3,3)$), let $U = \{12,21\}$. Then, $N[U] = \{10,20,01,11,21,02,12,22\}$, and $C \ominus U$ is a trivial graph that contains only the vertex $00$. Hence, $\kappa_{NB}(C) \leq |U| = 2$.

For other cases, let $v = 0000 \cdots 00$ and 
\[
U =
\begin{Bmatrix}
	1(d_2 - 1)00 \cdots 00,\\
	01(d_3 - 1)0 \cdots 00,\\
	001(d_4 - 1) \cdots 00,\\
	\cdots, \\
	0000 \cdots 1(d_n - 1),\\
	(d_1 - 1)000 \cdots 01
\end{Bmatrix}.
\]
Fig.~\ref{fig:upper} illustrates $v$, $N(v)$ and $U$. It is clear that $N(v) \subseteq N[U]$ and $v \notin N[U]$. Also, the number of vertices in the survival graph $C \ominus U$ is calculated as follows:  
\begin{equation*}
\begin{split}
|V(C \ominus U)| &\geq (\prod_{i=1}^{n} d_i) - n(2n + 1) \\
&\geq 
\begin{cases}
3^n - n(2n + 1) \geq 6 > 1 & \text{when}\ n \geq 3; \\
12 - n(2n + 1) = 2 > 1 & \text{when}\ n = 2\ \text{and there is a}\ d_j \geq 4\ \text{for}\ j\in\{1,2\}.
\end{cases}
\end{split}
\end{equation*}
Thus, $v$ is an isolated vertex in $C \ominus U$. Since $C \ominus U$ is not connected, we have $\kappa_{NB}(C) \leq |U| = n$.
\end{proof}

\begin{figure}[htbp]
	\centering
	\includegraphics[width=0.6\textwidth]{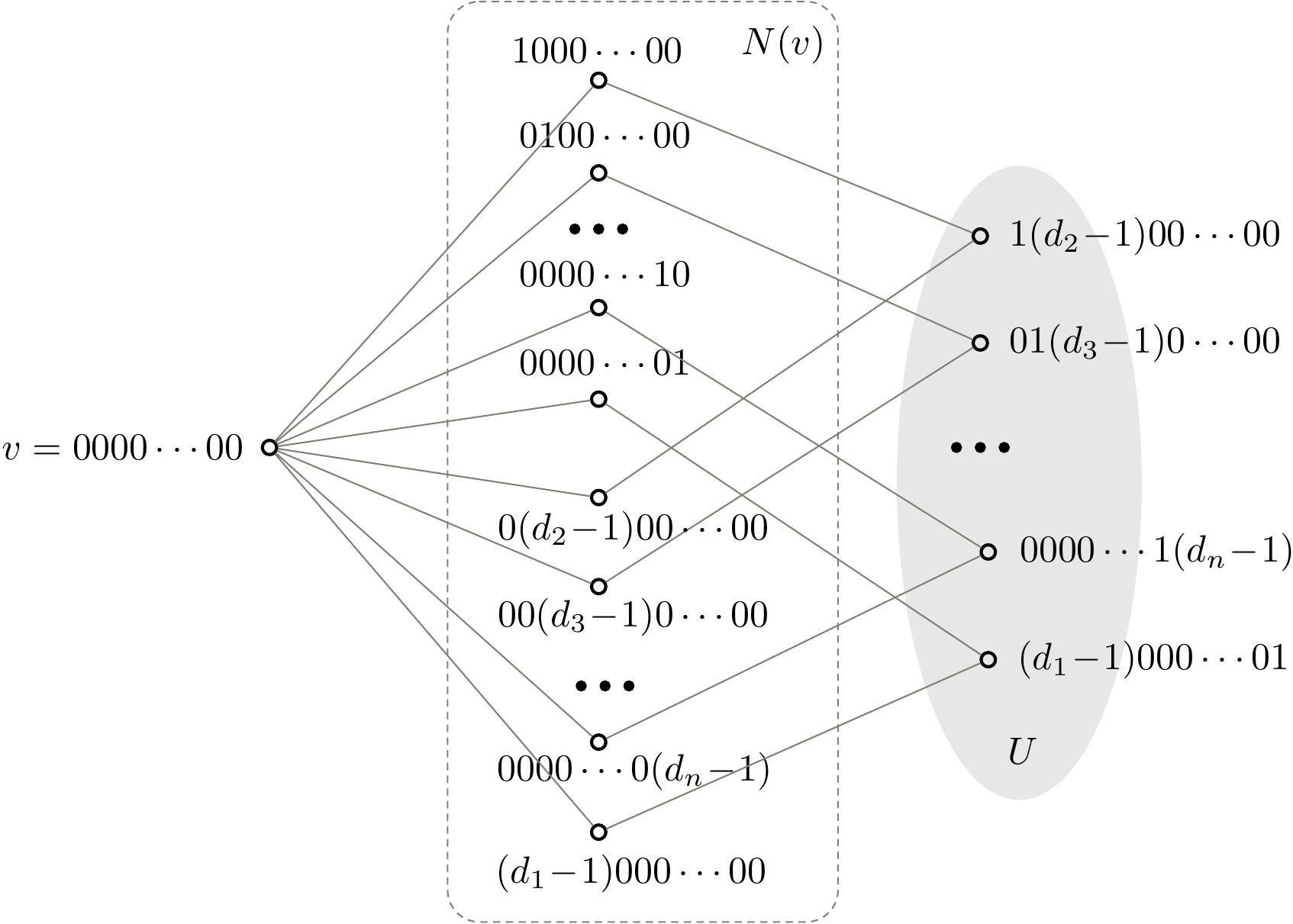}
	\caption{The illustration of $v$, $N(v)$ and $U$}
	\label{fig:upper}
\end{figure}

Combining Theorem \ref{thm:lower-bound} and Theorem \ref{thm:upper-bound}, we obtain the following main result.

\begin{theorem} \label{thm:main}
If $n \geq 2$ and $d_i \geq 3$ for $1 \leq i \leq n$, then $\kappa_{NB}(C(d_1, d_2, \ldots, d_n)) = n$.
\end{theorem}

As we mentioned in Section~\ref{preliminaries}, the $k$-ary $n$-cube $Q_n^k$ is a special case of the undirected toroidal mesh $C(d_1, d_2, \ldots, d_n)$, where $d_i = k$ for all $1 \leq i \leq n$. Therefore, we can derive a corollary regarding the neighbor connectivity of $Q_n^k$. This aligns with the findings presented in \cite{dvovrak2020neighbor} and expands that research work.

\begin{corollary} \label{corollary:neighbor connectivity of k-ary n-cube}
If $n \geq 2$ and $k \geq 3$, then $\kappa_{NB}(Q_n^k) = n$.
\end{corollary}

\section{Simulation and experimental results} \label{simulation}

In this section, we conduct a computer experiment to simulate random faults affecting vertices and their neighbors in undirected toroidal meshes, aiming to understand the situation in real network environments. Specifically, we select three undirected toroidal meshes of different scales: $C(3,4)$, $C(3,4,5)$, and $C(3,4,5,6)$.

Let the network participating in the trial be denoted as $G$ with the vertex set $V(G)$, and the set of all faulty source vertices in $G$ be represented as $U$. We will conduct one million independent trials for each undirected toroidal mesh $G$. Each trial consists of several rounds, depending on the status during the experimental process. Initially, all vertices in $G$ are healthy and set $U = \varnothing$. In the $i$-th round (where $i \geq 1$), a vertex $v$ is randomly selected from the set $V(G) \setminus U$, ensuring that each vertex has an equal probability of being chosen. Once selected, $v$ is designated as a faulty source vertex and added to the set $U$. Then, $v$ and all its adjacent vertices are removed from $G$. Finally, we assess whether the resulting survival graph $G \ominus U$ is disconnected, complete, or empty. If it is, we stop the trial, and the number of faulty source vertices is recorded as $\mu = |U|$. If not, the next round of the trial proceeds. The algorithm used in this trial is presented in the Appendix (see Algorithm~\ref{algo:Fault-test}).

After summarizing the experimental data, we present three bar charts, displayed in Figs.~\ref{fig:C34}, \ref{fig:C345}, and \ref{fig:C3456}. In these charts, the horizontal axis represents the number of faulty source vertices, while the vertical axis indicates how many times, out of one million trials, the graph \(G \ominus U\) becomes either disconnected, empty, or complete for the corresponding number of faulty source vertices shown on the horizontal axis.

In what follows, we examine each bar chart along with their respective experimental data. For convenience, we say that $G \ominus U$ reaches the \textit{target state} if it becomes a disconnected graph, an empty graph, or a complete graph. For $G=C(3,4)$, the full data from the simulation experiment is illustrated in Fig.~\ref{fig:C34}. This figure shows that when the number of faulty source vertices $|U| < 2$, there are no instances in one million trials where $G \ominus U$ achieves the target state. However, when $|U| = 2$, the resulting $G \ominus U$ has a 27\% chance of reaching the target state. According to Theorem~\ref{thm:main}, we know that $\kappa_{NB}(C(3,4)) = 2$, which indicates that our experiments are consistent with the theoretical value. Interestingly, in one million trials, the mean number of faulty source vertices needed to transform $G \ominus U$ into a target state is 2.95, with a median of 3 and a mode of 3. All these values exceed the theoretical value of $\kappa_{NB}(C(3,4))$.

\begin{figure}[htbp]
	\centering
	\includegraphics[width=10cm]{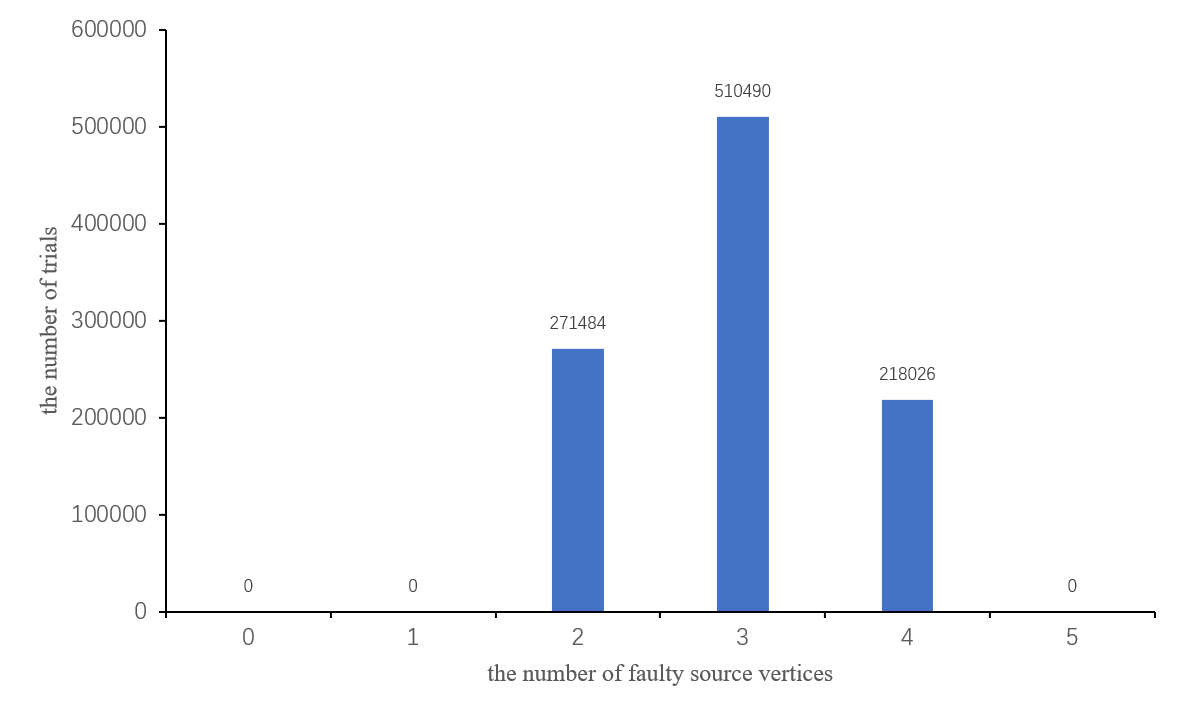}
	\caption{The distribution of the fault experiment on undirected toroidal mesh $C(3,4)$}
	\label{fig:C34}
\end{figure}

For $G=C(3,4,5)$, the complete data from the simulation experiment is presented in Fig.~\ref{fig:C345}. The results show that, in one million trials, the mean number of faulty source vertices required to bring $G \ominus U$ to the target state is 8.18, with a median of 8 and a mode of 7. The relationship between these statistics shows that the mode ($7$) $<$ median ($8$) $<$ mean ($8.18$), confirming that the distribution of faults is positively skewed. Additionally, by applying Theorem~\ref{thm:main}, we find that $\kappa_{NB}(C(3,4,5)) = 3$. As shown in Fig.~\ref{fig:C345}, we cannot find instances in one million trials when $|U| < 3$. However, when $|U| = 3$, there is a $1.7 \%$ chance that $G \ominus U$ will attain the target state. These findings are aligned with the theoretical predictions.

\begin{figure}[htbp]
	\centering
	\includegraphics[width=10cm]{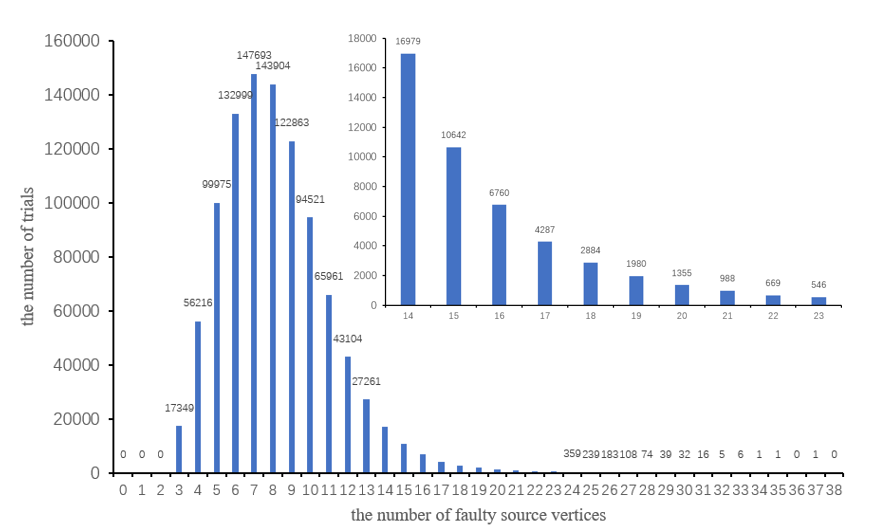}
	\caption{The distribution of the fault experiment on $C(3,4,5)$}
	\label{fig:C345}
\end{figure}

Finally, for the 4-dimensional undirected toroidal mesh $G=C(3,4,5,6)$, the simulation result is depicted in Fig.~\ref{fig:C3456}. The presentation of this experiment is similar to the previous one, maintaining entirely consistency with the theoretical values. After conducting one million of trials, the results yielded a mean of 26.18, a median of 26, and a mode of 26.

\begin{figure}[htbp]
	\centering
	\includegraphics[width=10cm]{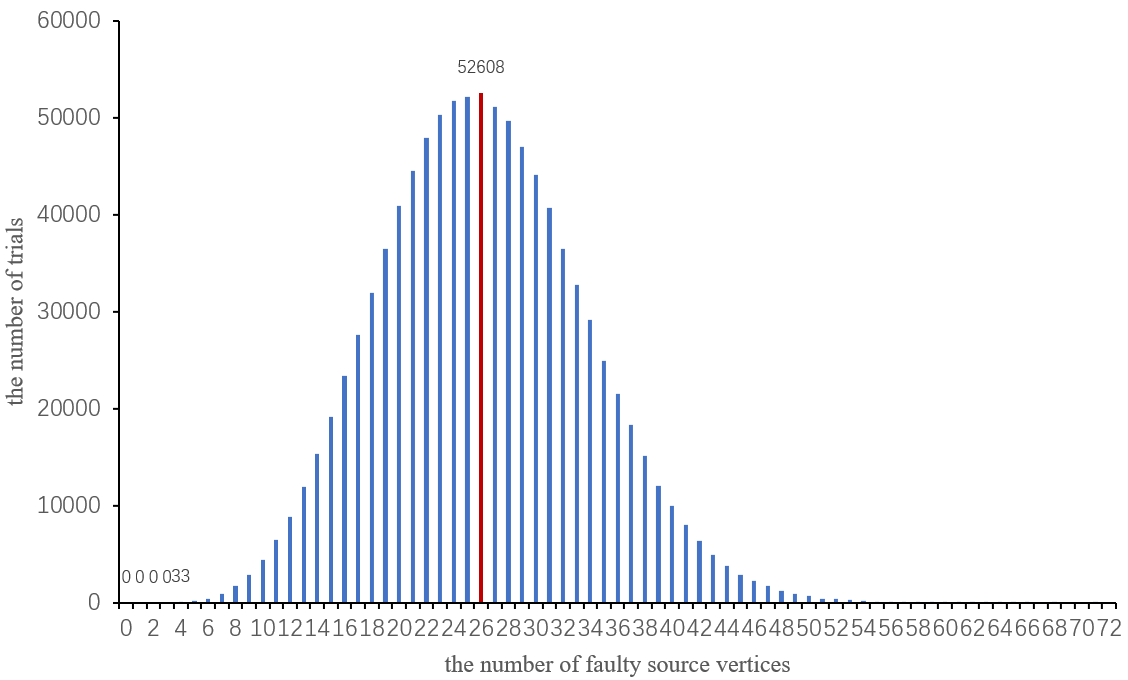}
	\caption{The distribution of the fault experiment on $C(3,4,5,6)$}
	\label{fig:C3456}
\end{figure}

\section{Concluding remarks} \label{concluding remarks}
In this paper, we determine the neighbor connectivity of the $n$-dimensional undirected toroidal mesh $C(d_1, d_2, \ldots, d_n)$ for $n \geq 2$ and $d_i \geq 3$, where $1 \leq i \leq n$. The connectivity is given by 
\[ 
\kappa_{NB}(C(d_1, d_2, \ldots, d_n)) = n. 
\]
This result generalizes the neighbor connectivity of the $k$-ary $n$-cube $Q_n^k$. Additionally, we provide an algorithm to randomly generate faulty source vertices in a simulated network, leading to faults affecting neighboring vertices. Our experiments ultimately confirmed that the experimental data were consistent with our theoretical findings.



\bibliographystyle{elsarticle-harv}
\parskip 0pt
\bibliography{main}

\end{document}